\documentclass[11pt,oneside]{amsart}

\usepackage[utf8]{inputenc}
\usepackage[english]{babel}
\usepackage[margin=3cm]{geometry} 
\usepackage{quiver}
\usepackage{amsmath,amsthm,amssymb}
\usepackage{mathtools}
\usepackage{graphicx}
\usepackage[colorlinks=true, allcolors=blue]{hyperref}
\usepackage{mathrsfs,calligra}
\usepackage{tikz-cd,adjustbox}
\usepackage{spverbatim,comment}
\usepackage[shortlabels]{enumitem}

\usepackage[backend=biber,style=alphabetic,maxalphanames=10,minalphanames=10, maxcitenames=99, , maxbibnames=99, doi=false,
  isbn=false,
  url=false,
  eprint=false, ]{biblatex}
\theoremstyle{plain}
\newtheorem{theorem}{Theorem}[section]
\newtheorem{lemma}[theorem]{Lemma}
\newtheorem{proposition}[theorem]{Proposition}
\newtheorem{corollary}[theorem]{Corollary}

\theoremstyle{definition}
\newtheorem{definition}[theorem]{Definition}
\newtheorem{example}[theorem]{Example}

\theoremstyle{remark}
\newtheorem{remark}[theorem]{Remark}

\newcommand{\Rom}[1]{\uppercase\expandafter{\romannumeral#1}}

\newcommand{\G}{\mathcal G}
\renewcommand{\H}{\mathcal H}

\newcommand{\X}{\mathcal X}
\newcommand{\Y}{\mathcal Y}

\newcommand{\CC}{\mathbb C}
\newcommand{\DD}{\mathbb D}

\newcommand{\HH}{\mathbb H}
\newcommand{\LL}{\mathbb L}

\newcommand{\QQ}{\mathbb Q}
\newcommand{\RR}{\mathbb R}
\newcommand{\TT}{\mathbb T}
\newcommand{\ZZ}{\mathbb Z}

\DeclareMathOperator{\im}{Im}

\DeclareMathOperator{\Ann}{Ann}
\newcommand{\xto}{\xrightarrow} 
\newcommand{\ito}{\hookrightarrow} 
\newcommand{\sto}{\twoheadrightarrow} 

\DeclareMathOperator{\sHom}{\mathscr{H}\text{\kern -3pt {\calligra\large om}}\,}
\newcommand{\Xhat}{\widehat{X}}      
\newcommand{\Yhat}{\widehat{Y}} 
\newcommand{\RHom}{{\bf R} \H om}
\DeclareMathOperator{\Ext}{Ext}

\newcommand{\DB}{\underline{\Omega}} 

\DeclareMathOperator{\Gr}{Gr}

\DeclareMathOperator{\depth}{depth}

\DeclarePairedDelimiter\abs{\lvert}{\rvert}

\makeatletter
\let\oldabs\abs
\def\abs{\@ifstar{\oldabs}{\oldabs*}}
\makeatother

\theoremstyle{definition} 
\newcommand{\thistheoremname}{}
\newtheorem*{genericthm}{\thistheoremname}

\makeatletter
\def\thm@space@setup{%
\thm@preskip=4pt plus 1pt minus 1pt 
\thm@postskip=4pt plus 1pt minus 1pt 
}
\makeatother

\newcommand{\Addresses}{{
  \vspace{\bigskipamount}
  \footnotesize
  
  \textsc{Department of Mathematics, Harvard University, 1 Oxford Street, Cambridge, MA 02138, USA}\par\nopagebreak
  \textit{E-mail address}: \texttt{astenie@math.harvard.edu}
}}

\title{Global smoothing of singular Fano and Calabi-Yau varieties}

\author{Anda Tenie}
\thanks{AT was partially supported by a Simons Dissertation Fellowship in Mathematics}

\bibliography{main.bib}
\begin{document}

\begin{abstract}
We study the problem of smoothing Fano and Calabi-Yau varieties with isolated Du Bois (equivalently, log canonical) lci singularities. For Fano varieties, we show that any such $Y$ admits a deformation to a Fano variety with only $1$-rational singularities, and if $Y$ has only non-$1$-rational hypersurface singularities, then $Y$ is smoothable. For Calabi-Yau varieties, we show first that any such $Y$ deforms to a Calabi-Yau with only $1$-Du Bois singularities. Moreover, if $Y$ has only non-$1$-Du Bois hypersurface singularities then $Y$ is smoothable. Both smoothing statements extend to lci singularities satisfying suitable local conditions. When we allow a Calabi-Yau to have $1$-liminal singularities, a global condition is needed, which we express in terms of the Hodge-Du Bois numbers of $Y$. These theorems recover and generalize results for threefolds of Friedman, Namikawa, Namikawa-Steenbrink, Gross, and Friedman-Laza. In higher dimensions, our results provide alternative smoothing conditions and also extend the work of Friedman-Laza from the case of rational hypersurface singularities to Du Bois lci singularities.
\end{abstract}

\maketitle

\section{Introduction}

In this paper, we study the problem of smoothing a Fano or Calabi-Yau variety $Y$ satisfying certain singularity conditions. The question of smoothing such varieties has been studied extensively in the case of threefolds \cite{friedman1986simultaneous}, \cite{namikawa1994deformations}, \cite{NamFano}, \cite{namikawa1995global}, \cite{gross1997deforming}, \cite{nam02} and, more recently, in higher dimensions \cite{friedman2025deformations}, \cite{friedmanlocal}, \cite{FLklim}, \cite{FL0lim}.
Constructing smoothings of singular Fano or Calabi-Yau varieties involves relating local deformation data to global data. Even when global deformations are unobstructed and each singularity admits a local smoothing, there is usually no guarantee that these local directions glue to a single global deformation that smooths all singular points. When a smoothing is out of reach, a closely related question is whether a variety can always be deformed to one with milder singularities. In this paper, we address both questions for varieties of dimension at least three with isolated Du Bois (or, equivalently, log canonical) local complete intersection singularities.

\begin{definition}\label{def: CYvar}
(i) A singular Fano variety $Y$ is a projective algebraic variety over $\CC$ of dimension $n\geq 3$ with isolated Du Bois 
local complete intersection singularities such that $\omega_Y^{-1}$
is ample. \newline
(ii) A Calabi-Yau variety $Y$ is a compact K\"ahler variety  (or, more generally, a certain compact complex analytic space; see Remark \ref{remark: analyticsp}) of dimension $n\geq 3$ with isolated Du Bois lci singularities such that $\omega_Y\cong \mathcal{O}_Y$ and $H^1(\mathcal{O}_Y)=0$ (see Remark \ref{rmk: projsm}).
\end{definition} 

\begin{definition}\label{def:smoothing}
We say that $Y$ is smoothable if there exists a proper flat analytic deformation
$\pi:\Y\to\Delta$ over the unit disk with special fiber
$\Y_0\simeq Y$ and smooth general fiber. 
\end{definition}

Our approach builds on the important work of Friedman-Laza, which revealed the role of higher rational and higher Du~Bois singularities in global smoothing problems. These notions refine the classical conditions of rational and Du~Bois singularities through Hodge theory. We focus on 1-rational and 1-Du Bois singularities, whose definitions and relevant properties are recalled in Section \ref{sec: 2.2}, particularly Definitions \ref{def: 1DB} and \ref{def: 1rat}.

Going from hypersurface to lci singularities introduces a further difficulty. Deformation directions that smooth a hypersurface may only split an lci singularity into milder singularities. To obtain smoothings, we introduce two local conditions, denoted $(\star_1)$ and $(\star_2)$ (see Definition \ref{def: star}), which ensure that the deformation directions supplied by our global arguments suffice to smooth the singularities. For hypersurface singularities, these conditions reduce to being non-1-Du Bois and non-1-rational, respectively.

For Calabi-Yau threefolds with isolated rational hypersurface singularities, a global condition is needed for smoothing and $\QQ$-factoriality provides such a hypothesis (see Theorem \ref{thm: NSQfact}). For higher dimensional Calabi-Yau varieties we introduce a global criterion which generalizes the $\QQ$-factoriality condition. We express it in terms of the Hodge-Du~Bois numbers of $Y,$ which are singular analogues of the Hodge numbers of a smooth projective variety, and are defined by
$$\underline{h}^{p,q}(Y) :=\dim_{\CC}\mathbb{H}^q (Y,\underline{\Omega}^{p}_{Y}),$$
where $\underline{\Omega}^{p}_{Y}$ denotes the $p$-th graded piece of the Du~Bois complex of $Y$ (see Section \ref{subsect: Hodge}).

\subsection{Statement of main results}
We begin by stating our results, and then discuss their consequences and relations to prior work. We start with the case when $Y$ is a Calabi-Yau variety. Interestingly, worse singularities (non-1-Du Bois) are easier to remove, as our first result illustrates.

\begin{theorem}\label{thm: introlcinot1DB}\textup{(Theorem \ref{thm: lcinot1DB})}
Suppose $Y$ is a singular Calabi-Yau variety with isolated Du Bois lci singularities. 
\begin{enumerate}
\item Then $Y$ can be deformed to a Calabi-Yau variety whose singularities are 1-Du Bois.
\item Assume $Y$ has only non-1-Du~Bois hypersurface singularities, or, more generally, that every singularity of $Y$ satisfies $(\star_1).$ Then $Y$ is smoothable, and every small deformation is Calabi-Yau.
\end{enumerate}
\end{theorem}

The additional local hypothesis in (2) corrects the corresponding smoothability statement in the previous version of this paper, where the non-1-Du Bois condition was asserted to suffice also in the lci setting. The earlier argument relied on \cite[Lemma 3.7]{gross1997deforming}, which turns out to contain an error (see Remark \ref{remark: 3foldcrepant}). In Example \ref{ex: notsmoothableglobal} we give a nonsmoothable Calabi-Yau variety with a unique isolated rational lci singularity which is not 1-Du Bois, showing that an additional local hypothesis such as $(\star_1)$ is needed for a general lci smoothing statement. The deformation statement in (1) requires no such additional hypothesis.

If one allows 1-liminal singularities (i.e. singularities that are 1-Du Bois but not 1-rational) one needs a global condition which guarantees the existence of a smoothing. We formulate a global condition in terms of the Hodge-Du Bois numbers of $Y$:

\begin{theorem}\label{thm: introlcigeneral}\textup{(Theorem \ref{thm: lcigeneral})}
Suppose $Y$ is a singular Calabi-Yau variety of dimension $n$ with isolated Du Bois lci singularities and assume that we have the following equality of Hodge-Du Bois numbers $\underline{h}^{n-1,2}(Y)=\underline{h}^{1,n-2}(Y).$
\begin{enumerate}
    \item If $Y$ is projective with $H^2(\mathcal{O}_Y)=0$ then it can be deformed to a projective Calabi-Yau whose singularities are 1-rational.
    \item Assume $Y$ has only non-1-rational hypersurface singularities, or, more generally, that every singularity of $Y $ satisfies $(\star_2).$ Then $Y$ is smoothable, and every small deformation is Calabi-Yau. 
\end{enumerate}
\end{theorem}

See the discussion after Corollary \ref{cor: 3foldODP}, which explains that the equality of Hodge-Du Bois numbers is a natural higher-dimensional generalization of prior smoothing conditions given by Namikawa-Steenbrink \cite{namikawa1995global} for threefolds. 

The methods developed in this paper also allow us to remove a technical assumption and to extend to the Du Bois setting a result of Friedman-Laza (see Theorem \ref{thm: 1limlci}).

\vspace{0.7em}
We now turn to the case when $Y$ is a Fano variety. We show:

\begin{theorem}\label{thm: introfanothm}\textup{(Theorem \ref{thm: fanothm})}
Suppose $Y$ is a singular Fano variety with isolated Du Bois lci singularities. 
\begin{enumerate}
    \item Then $Y$ can be deformed to a Fano variety whose singularities are 1-rational.
    \item Assume $Y$ has only non-1-rational hypersurface singularities, or, more generally, that every singularity of $Y$ satisfies $(\star_2).$ Then $Y$ is smoothable, and every small smoothing is Fano. 
\end{enumerate} 
\end{theorem}

\begin{remark}\label{remark: hypsmoothtotal}
In the context of all of the smoothing theorems above the stronger result holds that $Y$ admits a deformation over a smooth base with smooth total space. 
\end{remark}

We also show the following unobstructedness result, which has a simple proof using a Nakano-type vanishing theorem for Du Bois complexes:

\begin{theorem}\textup{(Theorem \ref{thm: FanoUnobstructed})}
Let $Y$ be a Fano variety with isolated lci singularities. Then $\TT^i_Y=0$ for $i\geq 2$ and, in particular, deformations of $Y$ are unobstructed.
\end{theorem}
For an extension outside of the lci case, see Theorem \ref{thm: unobstrnonlci}. 
\subsection{Related work}
For a history of the problem of smoothing Fano and Calabi-Yau threefolds we refer the reader to \cite{friedman2025deformations}. We mention here some theorems which are relevant to the results of this paper. For Calabi-Yau threefolds, Namikawa-Steenbrink show that $\QQ$-factoriality  is an appropriate global assumption that guarantees the existence of smoothings:

\begin{theorem}[\cite{namikawa1995global}]\label{thm: NSQfact}
 Suppose $Y$ is a projective Calabi-Yau threefold with rational isolated hypersurface singularities. Then the following hold:
 \begin{enumerate}
     \item $Y$ can be deformed to a Calabi-Yau with only ODPs.
     \item  Assume also that $Y$ is $\QQ$-factorial. Then $Y$ is smoothable.
 \end{enumerate}

\end{theorem}

Still for threefolds, in the case of lci singularities, Gross proves $Y$ is smoothable if it has a crepant resolution and no ODPs (ordinary double points). More precisely, he shows:

\begin{theorem}[\cite{gross1997deforming}]\label{thm: Grosslci}

Let $Y$ be a Calabi-Yau threefold with isolated rational lci singularities, and assume it admits a crepant resolution. Then there is a deformation of $Y$ which smooths all singular points of $Y$ except possibly the ordinary double points of $Y.$ In particular, if $Y$ has no ordinary double points, then $Y$ is smoothable.
\end{theorem}

\begin{remark}
The local smoothing criterion in \cite[Lemma 3.7]{gross1997deforming} requires correction (see Example \ref{example: notsmoothablelocal} for a counterexample). Nevertheless, Gross's theorem remains valid, as explained in Remark \ref{remark: 3foldcrepant} since its singularities that are not ODPs satisfy $(\star_1).$
\end{remark}
In higher dimensions, there has been recent progress on smoothing Calabi-Yau varieties with isolated rational hypersurface singularities. Indeed, Friedman-Laza \cite{friedman2025deformations} give smoothing results for Calabi-Yau varieties $Y$ whose singularities are either strongly 1-irrational or are 1-Du Bois but not 1-rational (also known as 1-liminal). Note that the strongly 1-irrational condition is stricter than the not 1-Du Bois condition, and it involves another technical assumption on the splitting of a certain short exact sequence (see \cite[Definition 2.6.(iv)]{friedman2025deformations}).  Denote by $Z'$ the union of the 1-liminal singularities and by $Z''$ the union of the strongly 1-irrational ones. Let $L$ be the union of the links of each singular 1-liminal point. If $x$ is a $1$-liminal hypersurface singularity, then by \cite[Section 4.11]{dimca2012koszul} (see also \cite[Corollary 6.14]{friedman2024higheriso}),  $\dim_{\CC} \Gr_F^{n-1}H^n(L_x) = 1$, and let $\varepsilon_x$ be a generator.

 \begin{theorem}\label{thm: FL}\textup{(\cite[Theorem 5.8.]{friedman2025deformations})} Let $Y' \to Y$ be a resolution over the points of $Z'$, so that the singular locus of $Y'$ is $Z''$. Let $\phi \colon   H^n(L) \to H_{n-1}(Y').$
  Finally, suppose  that, for every $x\in Z'$, there exists a nonzero $a_x\in \CC$   such that 
\begin{equation}\label{eq: homclasssum}
\sum_{x\in Z'}a_x\phi(\varepsilon _x) =0\in H_{n-1}(Y').
\end{equation}
 
Then there exists a global smoothing.   
 \end{theorem}

In fact, our methods strengthen this result by removing the technical splitting assumption and also by weakening the rationality hypothesis to the Du Bois condition. We also allow the non-1-Du Bois singularities to be lci, provided they satisfy $(\star_1)$ (see Theorem \ref{thm: 1limlci}).

So far, all the prior results mentioned assume the singularities are rational. We state the following recent result of Friedman-Laza which involves a Calabi-Yau whose unique singular point is Du Bois but not rational.

\begin{theorem}\label{thm: FL0lim}\textup{(\cite[Theorem 0.1]{FL0lim})}
Let $Y$ be a strict $0$-liminal Calabi--Yau variety (see \cite[Definition 3.1]{FL0lim}) whose unique singular point is an isolated hypersurface singularity. Assume that either
\begin{enumerate}
    \item the singularity is weighted homogeneous, or
    \item $Y$ is $3$-dimensional and $H^2(E_i,\mathcal{O}_{E_i})=0$ for every irreducible component $E_i$ of the exceptional divisor of a log resolution.
\end{enumerate}
Then $Y$ admits a smoothing.
\end{theorem}

\begin{remark}
The previous smoothing results use crucially that the deformations are unobstructed. In this generality this is a recent result of Friedman:
\end{remark}

\begin{theorem}\label{thm: unobstructed}\textup{(\cite[Theorem 1.6]{friedman2025unobstructed})}
Let $Y$ be a singular Calabi-Yau variety of dimension $n \ge 3$, such that the singularities of $Y$ are isolated, lci, and Du~Bois. Then the deformations of $Y$ are unobstructed.
\end{theorem}

Let us now mention relevant previous work in the Fano case. For Fano threefolds, Namikawa \cite{NamFano} showed that if $Y$ has Gorenstein terminal singularities then $Y$ is smoothable. This was recently generalized by Friedman-Laza to Fano varieties with isolated rational hypersurface singularities:

\begin{theorem}\label{thm: FanoFLsmoothing} \textup{(\cite[Corollary 4.11]{friedman2025deformations})}
Suppose that $Y$ is a Fano variety with only isolated rational hypersurface singularities,  such that one of the following holds: 
\begin{enumerate}
\item[(i)] $\dim Y = 3$; 
\item[(ii)] The singularities of $Y$ are $1$-irrational, there exists an element $H\in |\omega_Y^{-1}|$ with $H\cap Y_{\text{sing}} = \emptyset$, and   $H^3(Y, T^0_Y) =H^{n-3}(H, \Omega^1_H) =0$. 
\item[(iii)] $\dim Y \geq 4$, the singularities of $Y$ are $1$-liminal, there exists an element $H\in |\omega_Y^{-1}|$ with $H\cap Y_{\text{sing}} = \emptyset$, and   $H^{n-3}(H, \Omega^1_H) =0$.
\end{enumerate}
Then $Y$ is smoothable.

\end{theorem}

\subsection{Consequences of our results and relations to prior work}\label{sec: 1.3}
In dimension three, having 1-rational singularities simply means being smooth. Moreover, 1-liminal singularities are precisely ODPs.
Hence, our Theorem \ref{thm: introlcinot1DB} says:

\begin{corollary}\label{cor: 3foldCYnotODP}
Let $Y$ be a singular Calabi-Yau threefold with isolated Du Bois lci singularities. Then the following hold:
\begin{enumerate}
    \item $Y$ can be deformed to a Calabi-Yau threefold with only ODP singularities.
    \item If every singularity of $Y$ satisfies $(\star_1)$, then $Y$ is smoothable. In particular, this holds if $Y$ has only hypersurface singularities and no ODPs.
\end{enumerate}
\end{corollary}

Part (1) of the corollary recovers the first result of Namikawa-Steenbrink (Theorem \ref{thm: NSQfact}) and extends it to the Du Bois and lci setting. Our methods also imply Gross's smoothing result (Theorem \ref{thm: Grosslci}); see Remark \ref{remark: 3foldcrepant}.

If we allow ODPs and hope for a smoothing result, it is known that a global condition is necessary, as Namikawa \cite{namikawa1994deformations} gave an example of a Calabi-Yau threefold with one ordinary double point that is not smoothable. We have the following Corollary of Theorem \ref{thm: introlcigeneral}:

\begin{corollary}\label{cor: 3foldODP}
Suppose $Y$ is a singular projective Calabi-Yau threefold with isolated Du Bois lci singularities. Assume moreover that we have the following equality of Hodge-Du Bois numbers $\underline{h}^{2,2}(Y)=\underline{h}^{1,1}(Y).$ Then $Y$ is smoothable. 
\end{corollary}

For threefolds with rational singularities, Park and Popa \cite[Theorem E]{park2025q} show the equality $\underline{h}^{2,2}=\underline{h}^{1,1}$ holds if and only if $Y$ is $\QQ$-factorial (see also \cite[Theorem 3.2.]{namikawa1995global} which assumes rational isolated hypersurface singularities). Hence, Corollary \ref{cor: 3foldODP} recovers the second result in Theorem \ref{thm: NSQfact} of Namikawa-Steenbrink and, moreover, generalizes it to the case of Du Bois lci singularities.

In higher dimensions, one should think of our Theorem \ref{thm: introlcigeneral} as a generalization of Theorem \ref{thm: NSQfact}, where the condition $\underline{h}^{n-1,2}=\underline{h}^{1,n-2}$ is the natural replacement for $\QQ$-factoriality. The equality $\underline{h}^{n-1,2}=\underline{h}^{1,n-2}$ is related to Poincar\'{e} duality in the singular case. For a more detailed discussion of this, see Remark \ref{rem: PD}.\\

We now discuss how our results relate to prior work in the case when $Y$ is a Calabi-Yau of higher dimensions. In the case of lci singularities, to our knowledge, no prior results were known in dimension $\geq 4.$ 

In the $0$-liminal (Du Bois but not rational) case, the only prior result we are aware of is due to Friedman-Laza, whose smoothing results concern the case of a Calabi-Yau whose unique singular point is an isolated $0$-liminal hypersurface singularity (see Theorem \ref{thm: FL0lim}). Our results allow one to smooth $0$-liminal hypersurface points without their additional hypotheses, even when several such points or other milder singularities are present.

In terms of a global condition needed for smoothing 1-liminal points, Theorem \ref{thm: 1limlci} removes a technical assumption and extends to the Du Bois setting a result of Friedman-Laza (Theorem \ref{thm: FL}). We also give an alternative smoothing global condition, replacing a condition of type \eqref{eq: homclasssum} by the hypothesis on the Hodge-Du Bois numbers (see Theorem \ref{thm: introlcigeneral}). One can think of our stronger condition as implying the vanishing of an appropriate morphism, rather than the vanishing of a particular element as in Theorem \ref{thm: FL} or Theorem \ref{thm: 1limlci}. We hope this condition is more natural, as it directly relates to Poincar\'{e} duality. This is also useful in the lci setting, where our argument requires several independent local deformation directions and vanishing of the whole morphism allows all of them to lift.

We now move to the case of singular Fano varieties. The results in Theorem \ref{thm: introfanothm} give:
\begin{corollary}\label{cor: 3foldFano}
Suppose $Y$ is a singular Fano threefold with isolated Du Bois lci singularities. Then $Y$ is smoothable. 
\end{corollary}

Note that this recovers the previous threefold results of Namikawa and Friedman-Laza and, moreover, extends them to the Du Bois lci setting. There appear to be no previous results which include smoothing 0-liminal (Du Bois but not rational) singularities. We also give the first results in the setting of lci singularities. 

Returning to the case of hypersurface singularities, note that our Theorem \ref{thm: introfanothm} recovers part (i) of Theorem \ref{thm: FanoFLsmoothing} of Friedman-Laza and shows that none of the cohomological conditions in (ii) or (iii) are needed.

Finally, let us briefly comment on our Theorem \ref{thm: FanoUnobstructed}, which shows $\TT_Y^2=0$ (and, in particular, unobstructedness) for Fano varieties with isolated lci singularities. We note that this is new when $Y$ has Du Bois but not rational singularities (see Remark \ref{rmk: 0limunobstr}). Unobstructedness for Fano varieties was previously known in the lci 1-Du Bois (not necessarily isolated) case by Friedman-Laza \cite[Theorem 4.5]{friedman2025deformations} with a simple proof as in our work, which uses Nakano-type vanishing. Moreover, Friedman \cite[Theorem 3.17]{friedman2025unobstructed} recently proved unobstructedness for Fano varieties with rational isolated lci singularities by a different approach involving $T^1$-lifting (first proving an unobstructedness result in the log Calabi-Yau case). 

\textbf{AI Disclosure.} AI models were not used in previous versions of this article. In the current version AI models were used to construct Example \ref{example: notsmoothablelocal} and Example \ref{ex: notsmoothableglobal}. The paper was written by the author, who is responsible for its correctness.

\textbf{Acknowledgements.} I am grateful to Mihnea Popa for insightful conversations throughout the project and for constant support and encouragement during the preparation of this paper. I also thank Robert Lazarsfeld, Wanchun Shen, and Anh Duc Vo for many useful discussions, and Sung Gi Park for helpful suggestions that have sharpened some of the statements of this paper. I am especially thankful to Robert Friedman for useful discussions and for comments on a previous version of this paper, which have led to the inclusion of 0-liminal singularities in our results and to Theorem \ref{thm: 1limlci}. 
\section{Preliminaries}

\subsection{Hodge Theory of Singularities}\label{subsect: Hodge}
Let $Y$ be a complex algebraic variety. To extend the usual de Rham complex from the smooth case to singular varieties, fix a hyperresolution $\varepsilon_\bullet: Y_\bullet \to Y$. Following ideas of Deligne, Du~Bois defined $\underline{\Omega}_Y^\bullet := \mathbf{R}\varepsilon_{\bullet *}\Omega_{Y_\bullet}^\bullet$, an object in the derived category of filtered differential complexes on $Y$, and proved that it is independent of the chosen hyperresolution \cite{du1981complexe}. One can associate a filtration $ F^p\underline{\Omega}_Y^\bullet:=\mathbf{R} \varepsilon_{\bullet *} \Omega_{Y_\bullet}^{\geq p}$ by recalling that $\Omega_{Y_i}^\bullet$ is filtered by $\Omega_{Y_i}^{\geq p}.$ The $p$-th Du Bois complex of $Y$ is then $\underline{\Omega}_Y^p := \text{gr}_F^p \underline{\Omega}_Y^\bullet [p]$. For further background on Du~Bois complexes, see \cite{du1981complexe}, \cite[Chapter~V]{guillen2006hyperresolutions}, and \cite[§7.3]{peters2008mixed}.

We collect some properties of Du~Bois complexes
\begin{enumerate}
\item For every $p \ge 0$ there is a natural morphism $\Omega_Y^p \to \underline{\Omega}_Y^{p}$, which is an isomorphism when $Y$ is smooth. Here $\Omega_Y^p$ denotes the sheaf of K\"ahler differentials on $Y$; see \cite{du1981complexe} or \cite[p.~175]{peters2008mixed}. In particular, for all $i>0$, the cohomology sheaves $\mathcal{H}^i(\underline{\Omega}_Y^{p})$ are supported on the singular locus of $Y$.

\item We have $\underline{\Omega}^p_Y = 0$ for $p < 0$ and $p > n$, where $n = \dim Y$. Moreover $\H^q \underline{\Omega}^p_Y = 0$ if $p + q > n$ (see \cite[(4.1)]{steenbrink1985vanishing}).

\item There is a Hodge-to-de~Rham spectral sequence $E_1^{p,q}=\HH^q(Y,\underline{\Omega}_Y^{\,p}) \Rightarrow H^{p+q}(Y,\CC)$, which degenerates at $E_1$ when $Y$ is projective (see \cite[Thm.~4.5(iii)]{du1981complexe} or \cite[Prop.~7.24]{peters2008mixed}).

\item Let $Z \subset Y$ be a closed reduced subscheme, and let $\pi: \Yhat \to Y$ be a resolution of singularities that is an isomorphism over $Y \setminus Z$. Denote by $E := (\pi^{-1}(Z))_{\mathrm{red}}$ the reduced exceptional divisor, which is a simple normal crossing divisor. Then, by \cite[Ex.~7.25]{peters2008mixed}, for every $p$ there is an exact triangle
\[
\mathbf{R}\pi_{*}\Omega^p_{\Yhat}(\log E)(-E) \longrightarrow \underline{\Omega}_Y^{\,p} \longrightarrow \underline{\Omega}_Z^{\,p} \xto{+1}.
\]
In particular, if $Y$ has isolated singularities then $\mathbf{R}\pi_{*}\Omega^p_{\Yhat}(\log E)(-E)\simeq \underline{\Omega}_Y^{p}$ for $p>0.$
\end{enumerate}

\begin{definition}\label{def: hdb}
The \emph{Hodge-Du~Bois numbers} of $Y$ are defined by $$\underline{h}^{p,q}(Y) = \dim_{\mathbb{C}} \HH^q(Y, \underline{\Omega}_Y^p).$$
If $Y$ is smooth, then $\underline{\Omega}_Y^\bullet = \Omega_Y^\bullet$ and $\underline{h}^{p,q}(Y) = h^{p,q}(Y).$
\end{definition}

Let $\pi: \Yhat \to Y$ be a strong log resolution with exceptional divisor $E.$ If $x$ is a singularity, let $(X_x, x)$ denote its germ, where $X_x$ is a contractible Stein representative.
By restriction we have a strong log resolution $\pi: \Xhat_x \to X_x$ with exceptional divisor $E_x = \pi^{-1}(x).$ Let $U_x:=X_x \setminus \{x\}.$

Let us now restrict our attention to the case when $Y$ has \textit{isolated lci singularities}. As mentioned above, for $p>0$, $$\underline{\Omega}_Y^{p} \simeq \mathbf{R}\pi_{*}\Omega^p_{\Yhat}(\log E)(-E)$$

Then $\H^q(\underline{\Omega}_Y^{p})= R^q\pi_{*}\Omega^p_{\Yhat}(\log E)(-E),$ which for $q>0$ is supported on the singular points. Then $$H^0(R^q\pi_{*}\Omega^p_{\Yhat}(\log E)(-E))=\bigoplus\limits_{x \in \mathrm{Sing}(Y)} H^q(\Omega^p_{\Xhat_x}(\log E_x)(-E_x)),$$ since $X_x$ is a small Stein neighborhood around each singular point. The dimensions of these spaces are the so-called Du Bois invariants of the isolated singularity.

\begin{definition}\cite{steenbrink1997bois}
For $q>0,$ define the Du Bois invariants of an isolated singularity $(X,x)$ to be
$$b^{p,q} := \dim H^{q}(\Xhat ,\Omega^{p}_{\Xhat}(\log E)(-E)).$$ 
\end{definition}

\begin{theorem}[{\cite[Theorem 5]{steenbrink1997bois}}]\label{thm: duboisinv}
If $(X,x)$ is an isolated lci singularity, then $b^{p,q}=0$ for all $q>0$ unless $p+q\in\{n-1,n\}$.
\end{theorem}
Thus, by the theorem above, we see that in the case of isolated lci singularities if $q>0,$ we have $\H^q(\underline{\Omega}_Y^{p})=0$ except for $p+q\in \{n-1,n\}.$

Steenbrink also defines the link invariants:
$$\ell^{p,q}: = \dim H^q(E, \Omega_{\Xhat}^p(\log E)|_E)
= \dim \text{Gr}_F^p H^{p+q}(L).$$ The link $L$ of an $n$-dimensional isolated complete intersection singularity has $H^i(L)=0$ except for  $i \in \{0, n-1, n, 2n-1\}$. 
Hence, this immediately implies $\ell^{p,q} = 0$ unless $p+q \in \{0, n-1, n, 2n-1\}$.

\begin{remark}\label{rmk: nrvanish}
    Notice that using the restriction short exact sequence and the results on vanishing of $b^{p,q}$ and the vanishing of the cohomology of the link one immediately also has  that if $(X,x)$ is an isolated lci singularity then $$H^{q}(\Xhat ,\Omega^{p}_{\Xhat}(\log E))=0$$ for all $q>0$ unless $p+q\in\{n-1,n, 2n-1\}.$ 
\end{remark}

\subsection{Introduction to higher singularities}\label{sec: 2.2}

Because of their close connection to the minimal model program, rational and Du Bois singularities have been studied extensively. In particular, Kawamata log terminal singularities are rational, while log canonical singularities are Du Bois.
Advances in Hodge theory have prompted intensive study of the higher versions: $m$-Du Bois and $m$-rational singularities (see \cite{mustactua2023bois}, \cite{jung2022higher}, \cite{friedman2024higheriso}, \cite{friedman2024higher}, \cite{mustactua2025k}, \cite{QBM}, \cite{shen2023k}).
When $m=0$ one recovers the classical notions (rational and Du Bois), and as $m$ increases, the singularities get progressively milder, until they become smooth.

We begin by recalling the definitions of $m$-rational and $m$-Du Bois singularities, and then explain how these definitions become more manageable in our situation i.e., in the case of Du Bois isolated lci singularities. We focus specifically on the 1-rational and 1-Du Bois conditions, which will be useful notions for us in the context of local deformations.

Let $Y$ be a complex algebraic variety with lci singularities.

For each integer $p\ge 0$ there is a natural composition
$$
\Omega_Y^p \xlongrightarrow{\varphi_p}\underline\Omega_Y^p \xlongrightarrow{\phi_p} \DD_Y(\underline\Omega_Y^{n-p}).$$
We say that $Y$ has \emph{$m$-Du Bois} singularities if $$\varphi_p: \Omega_Y^p \to \underline\Omega_Y^p$$ is an isomorphism in
$D^b_{\mathrm{Coh}}(Y)$ for all $0\le p\le m.$

On the other hand, $Y$ is said to have $m$-rational singularities if the canonical morphisms 
$$\phi_p \circ \varphi_p: \Omega_Y^p \to \DD_Y(\underline\Omega_Y^{n-p})$$ are isomorphisms for all $0 \le p \le m$, where $\DD_Y(-) := \RHom(-,\omega_Y^\bullet)[-n]$.

If we moreover assume isolated singularities
then $Y$ has $m$-Du~Bois singularities if and only if $\Omega^p_Y \cong R^{0}\pi_{*}\Omega^{p}_{\Yhat}(\log E)(-E)$ and 
$$R^{q}\pi_{*}\Omega^{p}_{\Yhat}(\log E)(-E) = 0
\quad \text{for all } 0\leq p \leq m \text{ and all } q > 0. $$ Even more, if $n\geq 2m+1$ the former condition is automatically satisfied for lci singularities (see \cite[Lemma 3.19]{friedman2024higher}). A similar criterion holds for $m$-rational singularities. To summarize, we have:

\begin{definition}
Let $Y$ have isolated lci singularities and $\dim Y= n\geq 2m+1.$ Then $Y$ has $m$-Du Bois singularities if and only if $$
R^{q}\pi_{*}\Omega^{p}_{\Yhat}(\log E)(-E) = 0
\quad \text{for all } p \leq m \text{ and all } q > 0. $$ Also, $Y$ has $m$-rational singularities if and only if $$
R^{q}\pi_{*}\Omega^{p}_{\Yhat}(\log E) = 0
\quad \text{for all } p \leq m \text{ and all } q > 0. $$
Finally, we say a singularity is called $m$-liminal if it is $m$-Du Bois but not $m$-rational.
\end{definition}

We now explain the 1-rational and 1-Du Bois conditions under our assumptions.

Let $(X,x)$ be an isolated Du Bois lci singularity. Then by Theorem \ref{thm: duboisinv}, checking that $(X,x)$ is a 1-Du Bois singularity is equivalent to having $H^{n-1}(\Omega^{1}_{\Xhat}(\log E)(-E))=0$ and $H^{n-2}(\Omega^{1}_{\Xhat}(\log E)(-E))=0.$  But notice that the first equality follows from extra vanishing \cite[page 1369]{steenbrink1997bois}.

Hence, we have the following definition which will be used throughout the paper:
\begin{definition}\cite[Definition 2.6.]{friedman2025deformations}\label{def: 1DB} Let $(X, x)$ be an isolated Du 
Bois lci singularity.  Then we say  $(X, x)$ is 1-Du Bois if and only if $ H^{n-2}(\Xhat; \Omega^1_{\Xhat}(\log E)(-E)) = 0.$ By duality this is also equivalent to $H^2_E(\Omega^{n-1}_{\Xhat}(\log E))=0.$
\end{definition}

By Remark \ref{rmk: nrvanish} to check that $(X,x)$ is 1-rational is equivalent to $H^{n-1}(\Omega^{1}_{\Xhat}(\log E))=0$ and $H^{n-2}(\Omega^{1}_{\Xhat}(\log E))=0.$  We now explain why $(X,x)$ is 1-rational is equivalent to $H^{n-2}(\Omega^{1}_{\Xhat}(\log E))=0.$ If $H^{n-2}(\Omega^{1}_{\Xhat}(\log E))=0$ we first claim that $(X,x)$ is rational. Suppose otherwise that the singularity is Du Bois but not rational. Then consider the exact sequence 
$$H^{n-3}(\Omega^{1}_{\Xhat}(\log E)|_E)\to H^{n-2}(\Omega^{1}_{\Xhat}(\log E)(-E)) \rightarrow H^{n-2}(\Omega^{1}_{\Xhat}(\log E))\to H^{n-2}(\Omega^{1}_{\Xhat}(\log E)|_E).
$$

Note $H^{n-3}(\Omega^{1}_{\Xhat}(\log E)|_E)=\Gr^1_F H^{n-2}(L)=0$ since the singularity is isolated lci so we have an injection $H^{n-2}(\Omega^{1}_{\Xhat}(\log E)(-E)) \ito H^{n-2}(\Omega^{1}_{\Xhat}(\log E)).$ But since the singularity is not 1-Du Bois,  by Definition \ref{def: 1DB} we have $ H^{n-2}( \Omega^1_{\Xhat}(\log E)(-E)) \neq 0.$ This implies that $H^{n-2}( \Omega^1_{\Xhat}(\log E)) \neq 0,$ contradicting our assumption. Therefore, the singularity must be rational.
For an isolated rational singularity Musta{\c{t}}{\u{a}}-Olano-Popa \cite{mustactua2020local} show that $H^{n-1}(\Omega^{1}_{\Xhat}(\log E))=0$. Hence, the singularity is 1-rational.

Therefore, we have the following definition which will also be used throughout the paper:
\begin{definition}\cite[Definition 2.6.]{friedman2025deformations}\label{def: 1rat} Let $(X, x)$ be an isolated Du 
Bois lci singularity. We say $(X, x)$ is 1-rational if and only if $ H^{n-2}(\Xhat; \Omega^1_{\Xhat}(\log E)) = 0$.  By duality this is also equivalent to $H^2_E(\Omega^{n-1}_{\Xhat}(\log E)(-E))=0.$ We also say a singularity which is not 1-rational is 1-irrational.

\end{definition}

\subsection{Local deformations}\label{sec: 2.3}

Globally, one has a deformation space $\text{Def}(Y)$ whose tangent space is a finite-dimensional vector space $\TT^1
_Y=\Ext^1 (\Omega^1_Y,\mathcal{O}_Y)$ which classifies first-order deformations of $Y.$ Locally, for every singular point $x$ of $Y$ one can study the local deformation $\text{Def}(Y,x)$ of the germ $(Y,x).$ This has a tangent space given by the finite-dimensional vector space $T^1_{Y,x},$ where $T^1_Y$ is the sheaf $\mathcal{E}xt^1 (\Omega^1_Y,\mathcal{O}_Y),$ supported on the singular points. As before, $T^1_{Y,x}$ parameterizes first-order local deformations of the complex germ. There is a natural map $$\text{Def}(Y) \to \prod_{x \in Y_{\text{sing}}} \text{Def}{(Y,x)}$$ given by restricting to the germs. On
Zariski tangent spaces, this gives a homomorphism $$\TT^1_Y\xto{\eta}  H^0(T^1_Y)=\bigoplus\limits_{x \in \mathrm{Sing}(Y)}T^1_{Y,x}.$$
If $\eta$ were surjective we could lift local smoothing directions directly, but this is rarely the case; instead, we must identify suitable global classes whose images provide the required local smoothings.

In this section we focus on understanding the local deformations as well as criteria for global smoothing. We begin by recalling some of the results on deforming complex space germs. For more details see for example \cite[Chapter 2.1]{greuel25}

Let $(X,x)$ be the germ of an isolated lci singularity. The infinitesimal deformation theory of $(X,x)$ is controlled by the $\mathcal{O}_{X,x}$-modules
$$
T^i_{(X,x)} \coloneqq \text{Ext}^i_{\mathcal{O}_{X,x}}(\Omega^1_{X,x}, \mathcal{O}_{X,x}) \qquad (i = 1,2),
$$
where $T^1_{(X,x)}$ is the Zariski tangent space to the deformation functor and $T^2_{(X,x)}$ is the obstruction space. Since $(X,x)$ is lci, $T^2_{(X,x)} = 0$ and, in particular, the deformation functor is unobstructed.

Say $(X,x)$ has embedding dimension $e$, and is defined by  
$$f_1 = \cdots = f_k = 0, \qquad f_i \in \mathbb{C}\{x_1, \ldots, x_e\}.$$
Hence  
$$\mathcal{O}_{X,x} = \mathbb{C}\{x_1, \ldots, x_e\} / (f_1, \ldots, f_k).$$

The module of first-order deformations is  
$T^1_{(X,x)} \cong \mathcal{O}_{X,x}^{\oplus k} / J,$
where $J$ denotes the submodule of $\mathcal{O}_{X,x}^{\oplus k}$ generated by the columns  
$$\left( \frac{\partial f_1}{\partial x_i}, \ldots, \frac{\partial f_k}{\partial x_i} \right), \qquad i = 1, \ldots, e.$$

There exists an explicit miniversal deformation with smooth base space
$$
F:(\mathcal{X},x)\longrightarrow(S,0),\qquad 
\dim_{\mathbb{C}}S=\dim_{\mathbb{C}}T^1_{(X,x)}
$$
and the Kodaira-Spencer map gives the identification $T^1_{(X,x)}\cong T_{S,0}.$ 
Note that $(X,x)$ is smoothable. In fact, its discriminant locus is a reduced and irreducible hypersurface germ in $(S,0)$ (see \cite[page 645]{teissier1976hunting}). 

We now discuss criteria for determining when global deformations map to such smoothing directions.
\\
\paragraph{\textbf{Smoothing hypersurface singularities}.} In the case when $(X,x)$ is the germ of an isolated hypersurface singularity, one has a nice criterion for identifying when a one parameter deformation is a smoothing, formulated in terms of its Kodaira-Spencer class.  In this setting, the infinitesimal deformations of $(X,x)$ are parametrized by the module 
$$
T^1_{(X,x)} \cong \mathcal{O}_{X,x} / J,
$$
where $J = \left( \frac{\partial f}{\partial x_1}, \ldots, \frac{\partial f}{\partial x_e} \right)$ is the Jacobian ideal associated to the defining equation $f \in \mathbb{C}\{x_1, \ldots, x_e\}$ of $X$. In particular, $T^1_{X,x}$ is a cyclic $\mathcal{O}_{X,x}$-module. The following tells us that a deformation has smooth total space if the Kodaira-Spencer class is a generator of $T^1_{X,x}$ as an $\mathcal{O}_{X,x}$-module.

\begin{proposition}\label{prop: cyclicmod}
    Let $\X' \to \Delta$ be a one-parameter deformation of $(X, x).$ Then $\X'$ is smooth in a neighborhood of $x$ if and only if the Kodaira-Spencer class $u \in T^1_{X,x}$ is not in  $m_xT^1_{X,x},$ if and only if $u$ is a generator of $ T^1_{X,x}$ as an $\mathcal{O}_{X,x}$-module.
\end{proposition}

\begin{proof}
    See, for example, \cite[Lemma 1.9]{friedman2025deformations}.
\end{proof}

\begin{remark}
    Note that, in particular, since the total space is smooth this deformation is a smoothing of $(X,x).$ 

\end{remark}
\paragraph{\textbf{Deforming lci singularities}.} In the case of isolated lci singularities, to obtain a smoothing criterion one will need a stronger assumption (see Lemma \ref{lemma: smoothinglci}). However, to deform our variety to one with simpler singularities we will use the following lemma which again only requires finding one element which maps outside of $m_xT^1_{X,x}.$

Let $(X, 0)$ be the germ of an isolated lci singularity and $$F' : (\mathcal X', 0) \to (S', 0)$$ be a flat deformation of $(X, 0)$, where $S'$ is smooth at $0 \in S'$ with tangent space $T = T_{S',0}$. Since $F$ is a miniversal deformation there exists a morphism of germs $\phi:(S',0)\to(S,0)$ such that $F'$ is isomorphic to the pullback of $F$ along $\phi$. Under the identification $T_{S,0}\cong T^1$, we obtain a unique differential
$$d\phi_0:T=T_{S',0}\longrightarrow T_{S,0}\cong T^1 $$ which is the Kodaira-Spencer map of $F'$.

\begin{lemma}\label{lemma: deforminglci}
 If $\im(T \to T^1) \not\subset m T^1,$ then either
\begin{enumerate}
    \item a general fiber of  $X' \to S'$ is smooth or
    \item a general fiber of  $X' \to S'$ has at least two singular points.
\end{enumerate}

\end{lemma}

\begin{proof}[ Proof of Lemma~\ref{lemma: deforminglci}]
 
Following \cite{teissier1976hunting} we let $D \subset (S,0)$ which is the discriminant locus of the miniversal deformation $F:(\mathcal{X},0)\to (S,0)$ and $C$ the critical locus of $F.$  Teissier shows that $F|_C:C\to D$ is a normalization map.

We assume by contradiction that neither (1) nor (2) holds. In other words, we assume that the general fiber of $F'$ has exactly one singular point. Since every fiber of $F'$ is singular, the image of $\phi$ is contained set-theoretically in $D.$ As $S'$ is reduced, $\phi$ therefore factors through $D$ as a morphism of complex spaces.

We now claim that $S'\to D$ must factor through $C.$ To see this note $C'=C\times_S S' \to S'$ is finite as the base change of a finite map, and it is dominant since a general fiber of $F'$ is singular. Take $C_0$ an irreducible dominating component and then take its reduced structure. Then $C_0\to S'$ is a dominant finite map. By our assumptions, it is also of degree one since a general fiber has exactly one singular point. It is therefore birational, from a reduced and irreducible space to a normal one. Hence, it must be an isomorphism.

The inverse of $C_0\to S'$, composed with the natural map $C_0\to C$ and the inclusion $C\subset\mathcal X$, gives a map $\psi:(S',0)\to(\mathcal X,0)$ such that $\phi= F \circ \psi.$ Hence $d\phi_0 =dF_0 \circ d\psi_0.$

Using the presentation of $T^1$ from earlier in this section, choose a basis consisting of the constant vectors $e_1,\ldots,e_k$ and classes represented by vectors $g_1,\ldots,g_r$ vanishing at the origin. The miniversal deformation has equations $$f_i(x)+a_i+\sum_{j=1}^r b_jg_{ji}(x)=0$$ where $i=1,\ldots,k.$ Since the $f_i$ have no linear terms and the $g_{ji}$ vanish at the origin, the tangent space equations are $da_1=\cdots=da_k=0.$ Thus $\im(dF_0)$ is the subspace spanned by the $b$-directions, which corresponds to $mT^1.$ Therefore, $$\im(T\to T^1)=\im(d\phi_0)\subset\im(dF_0)=mT^1,$$ which gives a contradiction.
\end{proof}

\begin{remark}
In \cite[Lemma 3.7]{gross1997deforming}, it is claimed that only case (1) in \ref{lemma: deforminglci} can occur. Example \ref{example: notsmoothablelocal} below provides a counterexample. The gap in the argument is the following. As before, let $D \subset S$ be the discriminant of the miniversal deformation and $C$ its critical locus, so that $F|_C : C \to D$ is the normalization of $D$. If the general fiber of $F'$ is singular, then $\phi(S') \subset D$. Now it is claimed in \cite{gross1997deforming} that $S'\to D$ must factor through $C$ because $S'$ is smooth. However, smoothness alone does not guarantee the existence of a lift. 
The universal property of normalization guarantees such a
factorization for dominant maps from normal spaces, but $\phi(S')$ may lie entirely in the nonnormal locus of $D$.
For such maps, a lift need not exist (see Example \ref{example: notsmoothablelocal}). When the general fiber has a unique singular point, such a lift exists, as we showed in Lemma \ref{lemma: deforminglci}.
\end{remark}

\begin{example}\label{example: notsmoothablelocal}
This example is a counterexample to \cite[Lemma 3.7]{gross1997deforming}. Let $f_1 = x_1^2 + x_2^2 + x_3^2 + x_4^2$ and
$f_2 = x_1^2 + 2x_2^2 + 3x_3^2 + 4x_4^2 + 5x_5^2.$ Consider the smooth del Pezzo surface $S=\{f_1=f_2=0\}$ in $\mathbb P^4$ and let $X_0$ be the cone over $S.$  Note that $X_0$ has an isolated rational lci singularity. Consider the one parameter deformation $X_t=\{f_1=f_2=t\}$ and note the image in $T^1$ is spanned by the class of $(1,1)$, which is not in $mT^1.$ Therefore, the assumptions of Lemma \ref{lemma: deforminglci} are satisfied. For general $t$, the fiber $X_t$ has exactly two ordinary double points at $(\pm\sqrt{t},0,0,0,0)$. Thus this deformation is not a smoothing and case (2) of Lemma \ref{lemma: deforminglci} can occur.
\end{example}

\begin{remark}
Despite the issue with \cite[Lemma 3.7]{gross1997deforming}, Gross’s main result from that section \cite[Theorem 3.8]{gross1997deforming} remains valid (see Remark \ref{remark: 3foldcrepant}).
\end{remark}

\paragraph{\textbf{Smoothing lci singularities}.}

We now give a sufficient criterion for smoothing isolated lci singularities. The idea is to require that the image of the Kodaira-Spencer map generate the whole quotient $T^1_{X,x}/m_xT^1_{X,x}$, rather than merely contain an element outside $m_xT^1_{X,x}.$ In the hypersurface case this quotient is one-dimensional, so these two conditions coincide, recovering Proposition \ref{prop: cyclicmod}.

Consider notation as written before Lemma \ref{lemma: deforminglci}.

\begin{lemma}\label{lemma: smoothinglci}
 If the Kodaira-Spencer map $T\to T^1$ induces a surjection $T\twoheadrightarrow T^1/mT^1,$ then $\mathcal X'$ is smooth in a neighborhood of the point. In particular, a general fiber of $F'$ is smooth.
\end{lemma}

\begin{proof}
The proof is essentially the same as in the hypersurface case \cite[Lemma 1.9]{friedman2025deformations}. Using the notation from earlier in this section, write the deformation as $$\mathcal X'=V(G_1,\ldots,G_k)\subset\mathbb C^e\times S',$$ where $G_i(x,0)=f_i(x).$ Choose coordinates $s_1,\ldots,s_\ell$ on $S'.$ Since $e$ is the embedding dimension, the $f_i$ have no linear terms. Hence $J\subset m\mathcal O_{X,x}^{\oplus k}$, and the presentation above gives $T^1/mT^1 \cong\mathbb C^k.$ Under this identification, the induced Kodaira-Spencer map is represented by the matrix $$A=\left(\frac{\partial G_i}{\partial s_j}(0,0)\right).$$ By surjectivity, $A$ has maximal rank $k.$ The Jacobian criterion therefore shows that $\mathcal X'$ is smooth near the singular point. In particular, a general fiber is smooth.
\end{proof}

\paragraph{\textbf{Exact sequences in local deformations.}} Let $(X,x)$ be the germ of an isolated Du Bois lci singularity. 

In order to lift local first-order deformations to global ones it is often useful to look at appropriate quotients of $T^1_{X,x}$ which fit into certain global long exact sequences. To do this, we consider the identification $H^0(X,T^1_X)\cong H^2_x(X, T^0_X) \cong H^1(T_{U})$ due to Schlessinger \cite[Theorem 2]{schlessinger1971rigidity}.

If we consider the inclusions $\Omega^{n-1}_{\Xhat}(\log E)(-E) \subseteq \Omega^{n-1}_{\Xhat} \subseteq \Omega^{n-1}_{\Xhat}(\log E)$ notice that they all restrict to $\Omega^{n-1}_{U}$ on the smooth locus $U,$ which we have identified with $\Xhat \setminus E$. Therefore, one can identify the first-order deformations with  $$H^1(U; \Omega^{n-1}_{\Xhat}|_U), \; H^1(U; \Omega^{n-1}_{\Xhat}(\log E)|_U), \; \text{or} \; H^1(U; \Omega^{n-1}_{\Xhat}(\log E)(-E)|_U)$$ which all fit into local cohomology long exact sequences.

 Following Friedman-Laza, we recall without proof some of the results we will need for the local deformation theory. For more details one should consult \cite[Section 2]{friedman2025deformations} and \cite[Section 2]{FL0lim}. See also \cite[Section 1]{namikawa1995global}.

There exists a short exact sequence 
$$0\to  H^1(\Omega^{n-1}_{\Xhat}(\log E)) \to  H^1(T_{U}) \to H^2_{E}(\Omega^{n-1}_{\Xhat}(\log E))\to 0.$$ Hence, in the case when $H^2_{E}(\Omega^{n-1}_{\Xhat}(\log E))\neq 0$, i.e. when the singularities are not 1-Du Bois (see Definition \ref{def: 1DB}),  to construct global smoothings the idea will be to consider appropriate elements in $H^2_{E}(\Omega^{n-1}_{\Xhat}(\log E))$ that come from global first-order deformations. For the more general case of 1-irrational singularities (see Definition \ref{def: 1rat}), one looks at the image in $H^2_{E}(\Omega^{n-1}_{\Xhat}(\log E)(-E))$ as we explain below.

Then there exists an exact sequence coming from local cohomology:
$$ H^1(\Omega^{n-1}_{\Xhat}(\log E)(-E)) \to  H^1(T_{U}) \to H^2_{E}(\Omega^{n-1}_{\Xhat}(\log E)(-E))\to 0.$$
and, moreover, $H^2_{E}(\Omega^{n-1}_{\Xhat}(\log E)(-E))$ sits in the following exact sequence:
$$0\to \text{Gr}_F^{n-1}H^n(L) \to H^2_{E}(\Omega^{n-1}_{\Xhat}(\log E)(-E))\to H^2_{E}(\Omega^{n-1}_{\Xhat}(\log E))\to 0,$$ where $L$ is the link of the singularity.

Finally, if $(X,x)$ is 1-liminal note that $\text{Gr}_F^{n-1}H^n(L) \cong H^2_{E}(\Omega^{n-1}_{\Xhat}(\log E)(-E)).$ From the following commutative diagram 
\[\begin{tikzcd}[column sep=0.7em]
	{0=H^1_{E}(\Omega^{n-1}_{\Xhat}(\log E))} & {H^1(\Omega^{n-1}_{\Xhat}(\log E))} & {H^1(\Omega^{n-1}_{U})} & {H^2_{E}(\Omega^{n-1}_{\Xhat}(\log E))=0} & {} \\
	0 & {H^1(\Omega^{n-1}_{\Xhat}(\log E)(-E))} & {H^1(\Omega^{n-1}_{U})} & {H^2_{E}(\Omega^{n-1}_{\Xhat}(\log E)(-E))} & 0
	\arrow[from=1-1, to=1-2]
	\arrow[from=1-2, to=1-3]
	\arrow[from=1-3, to=1-4]
	\arrow[equals, from=1-3, to=2-3]
	\arrow[from=1-4, to=1-5]
	\arrow[from=2-1, to=2-2]
	\arrow[from=2-2, to=1-2]
	\arrow[from=2-2, to=2-3]
	\arrow[from=2-3, to=2-4]
	\arrow[from=2-4, to=1-4]
	\arrow[from=2-4, to=2-5]
\end{tikzcd}\]

so we find that in the 1-liminal case one has a short exact sequence  $$0\to H^1(\Omega^{n-1}_{\Xhat}(\log E)(-E)) \to H^1(\Omega^{n-1}_{\Xhat}(\log E)) \to H^2_{E}(\Omega^{n-1}_{\Xhat}(\log E)(-E)) \to 0.$$

\subsection{Local conditions for global smoothings}

We now explain how smoothing results will be obtained using Lemma \ref{lemma: smoothinglci} and the local cohomology results in the previous section. When $Y$ is a singular Calabi-Yau or Fano variety, the idea will be to apply this to $S'=\text{Def}(Y)$ which is smooth by the unobstructedness (Theorem \ref{thm: unobstructed} and Theorem \ref{thm: FanoUnobstructed}).

By Lemma \ref{lemma: smoothinglci}, to obtain a smoothing it is sufficient to show that, for each singular point $x$, the natural map $\TT^1_Y\to T^1_{Y,x}/m_xT^1_{Y,x}$ is surjective. Thus we will look for global deformation classes whose images generate these local quotients. Our global arguments will give for each singular
point $x\in Y$ a surjection from $\TT^1_Y$ onto a local cohomology quotient $K_x$ of $T^1_{Y,x}$ i.e we will have a composition $$\TT^1_Y\to T^1_{Y,x}\to K_x$$ which will be surjective.
We introduce two local conditions ensuring that the surjection onto $K_x$ gives the required surjection onto $T^1_{Y,x}/m_xT^1_{Y,x}$.

Let $(X,x)$ be an isolated Du Bois lci singularity of dimension $n\geq 3$, and consider the notation of Section \ref{sec: 2.3}. Set
$$K_1:=H^2_E(\Omega_{\Xhat}^{n-1}(\log E)),\text{ and } K_2:=H^2_E(\Omega_{\Xhat}^{n-1}(\log E)(-E)).$$
The local cohomology sequences recalled above give natural surjections $s_i:T^1_{X,x}\twoheadrightarrow K_i,$ for $i=1,2.$

\begin{definition}\label{def: star}
We say that a singularity $(X,x)$ satisfies condition $(\star_i)$ if $s_i$ induces an
isomorphism
$$T^1_{X,x}/m_xT^1_{X,x}
\xrightarrow{ \sim } K_i/m_xK_i.$$
\end{definition}

By the local cohomology sequences, this says that the image of $H^1(\Omega_{\Xhat}^{n-1}(\log E))\to T^1_{X,x}$ is contained in $m_xT^1_{X,x}$ when $i=1$, and that the image of $H^1(\Omega_{\Xhat}^{n-1}(\log E)(-E))\to T^1_{X,x}$ is contained in $m_xT^1_{X,x}$ when $i=2$.

\begin{proposition}\label{prop: star} Let $(X,x)$ be an isolated Du Bois lci singularity of dimension $n\geq 3$.
Then:
\begin{enumerate}
\item Condition $(\star_1)$ implies that $(X,x)$ is not $1$-Du Bois, and condition $(\star_2)$ implies that $(X,x)$ is not $1$-rational.
\item If $(X,x)$ is a hypersurface singularity, then
$(\star_1) \iff (X,x)$  is not 1-Du Bois and  $(\star_2) \iff (X,x)$ is not 1-rational.
\end{enumerate}
\end{proposition}

\begin{proof}
Since $(X,x)$ is singular and lci, $T^1_{X,x}/m_xT^1_{X,x}\neq 0$. Thus $(\star_i)$ implies $K_i\neq 0$, which gives (1). For a hypersurface singularity, $T^1_{X,x}/m_xT^1_{X,x}$ is one-dimensional. If $K_i\neq 0$, then $K_i/m_xK_i\neq 0$ by Nakayama's
lemma. The surjection between these two spaces is therefore an isomorphism, proving (2).    
\end{proof}

Let us now explain what these conditions mean in the lci case for a weighted homogeneous singularity in which case one can make this explicit in terms of degrees and weights.

For the rest of this section we let $X=\{f_1=\cdots=f_c=0\}\subseteq\mathbb C^{n+c}$ be an isolated weighted homogeneous complete intersection Du Bois singularity at the origin. Assume this is a minimal embedding, so the $f_i$ have no linear terms. Let $a_1,\dots, a_{n+c}$ be the positive integer weights of the coordinates and let $d_1\leq \dots \leq d_c$ be the weighted degrees of the defining equations.

Then $(\star_1)$ is equivalent to $$\sum_{j=1}^{n+c}a_j<\sum_{i=1}^c d_i+d_1.$$ Indeed, if we set $N=\sum_j a_j-\sum_i d_i$ then by \cite[Theorem 3.3 and Remark 3.4]{friedman2025deformations}, the image of $H^1(\Omega_{\Xhat}^{n-1}(\log E))$ in $T^1_{X,0}$ consists precisely of the weight spaces of weight $\geq -N$.

Note $T^1_{X,0}/m_0T^1_{X,0}$ has a basis given by the constant vectors of weights $-d_i$. Hence $$(\star_1)\iff -d_i<-N \text{ for all }i \iff N<d_1, $$
as claimed.
We now understand condition $(\star_1)$ in a few examples:
\begin{example}
(1) If $c=1$ i.e. if we have a  hypersurface singularity the inequality becomes $\sum_{j=1}^{n+1}a_j<2d.$
Recall that for an isolated weighted homogeneous hypersurface singularity, the minimal exponent is given by
$$\widetilde{\alpha}=\frac{\sum_{j=1}^{n+1}a_j}{d}.$$ So the above inequality is equivalent to $\widetilde{\alpha}<2$, which is precisely the condition that $(X,0)$ is not 1-Du Bois which agrees with Proposition \ref{prop: star}.\newline
(2) If all the degrees are equal $d_i=d$ then the minimal exponent is $\widetilde{\alpha}=\frac{\sum_{j=1}^{n+c} a_j}{d}$ and $(\star_1)$ is equivalent to $\sum_{j=1}^{n+c} a_j<(c+1) d.$ Therefore $(\star_1)\iff  \widetilde{\alpha}<c+1 \iff $ the singularity is not 1-Du Bois.

\end{example}

In general, for weighted homogeneous isolated complete intersection singularities with possibly different degrees, Chen-Dirks-Olano \cite[Corollary D]{MR5078859} prove that
$$c+\left\lfloor\frac{N}{d_c}\right\rfloor \leq  \widetilde{\alpha} \leq c+ \frac{N}{d_c}.$$
In particular, they obtain equality $\widetilde{\alpha}=c+\frac{N}{d_c}$ whenever $d_c$ divides $N$, and they expect this equality to hold in general. Assuming their expected equality, the characterization of $(\star_1)$ becomes
$$(\star_1)\iff N<d_1 \iff \widetilde{\alpha}< c+\frac{d_1}{d_c}.$$

\begin{remark}\label{remark: 3foldcrepant}
We claim that despite the issue with \cite[Lemma 3.7]{gross1997deforming}, Gross’s main smoothing result in the lci case \cite[Theorem 3.8]{gross1997deforming} still holds i.e. that if $Y$ is a singular Calabi-Yau 3-fold admitting a crepant resolution and if $Y$ has rational, isolated lci singularities then there is a deformation of $Y$ which smooths all singular points of $Y$ except possibly the ordinary double points.

The idea is to check that the singularities in question satisfy $(\star_1).$ Once this is shown one can follow for instance the proof of Theorem \ref{thm: lcinot1DB} and apply Lemma \ref{lemma: smoothinglci} to show the singularities that are not ODP (or, equivalently, not 1-Du Bois) can be smoothed.

We now explain why a threefold isolated lci rational singularity $(X,x)$ admitting a crepant resolution and which is not a node satisfies $(\star_1).$ Consider the exact sequence $$0\to H^1(\Omega^{n-1}_{\Xhat}(\log E)|_E) \to K_2=H^2_{E}(\Omega^{n-1}_{\Xhat}(\log E)(-E))\xrightarrow[]{\phi} K_1=H^2_{E}(\Omega^{n-1}_{\Xhat}(\log E))\to 0.$$ Then note that $m_x$ acts as zero on $\Omega^{n-1}_{\Xhat}(\log E)|_E$ and so on  $H^1(\Omega^{n-1}_{\Xhat}(\log E)|_E).$ This implies that $\text{Soc}(K_2)=\Ann_{K_2}(m_x)$ contains $ \ker \phi= H^1(\Omega^{n-1}_{\Xhat}(\log E)|_E).$ 

We now make use of \cite[Lemma 3.6]{gross1997deforming}  and its proof. Let $\rho:\widetilde X\to X$ be the crepant resolution and set $$T':=T^1_{X,x}/\text{im}(H^1(T_{\widetilde X})\to T^1_{X,x}).$$ Fix a local generator of $\omega_X$, and now by crepancy identify
$T_{\widetilde X}\cong\Omega^{n-1}_{\widetilde X}$, so $T'$ is also the quotient by the image of $H^1(\Omega^{n-1}_{\widetilde X})$. Now by functoriality of the Du Bois complex the restriction map factors as
$$H^1(\Omega^{n-1}_{\Xhat}(\log E)(-E)) \to H^1(\Omega^{n-1}_{\widetilde X}) \to H^1(\Omega^{n-1}_U)\cong T^1_{X,x}$$ and so the image defining $K_2$ is contained in the image defining $T'$. Hence there are compatible surjections $T^1_{X,x}\twoheadrightarrow K_2\twoheadrightarrow T'.$
By proof of Lemma 3.6 in \cite{gross1997deforming} we know $T^1_{X,x}/m_xT^1_{X,x} \xrightarrow{\sim} T'/m_xT'$ and $\text{Soc}(T')\subset m_xT'.$ Note that this immediately implies $T^1_{X,x}/m_xT^1_{X,x} \xrightarrow{\sim} K_2/m_xK_2$ i.e the singularity satisfies $(\star_2)$ and we claim that also $T^1_{X,x}/m_xT^1_{X,x} \xrightarrow{\sim} K_1/m_xK_1$ i.e the singularity satisfies $(\star_1).$ To see this note that the $\ker \phi \subset \text{Soc}(K_2).$ This maps to $\text{Soc}(T')$ which by Lemma 3.6 in \cite{gross1997deforming} is contained in $m_xT'.$ Therefore, following the isomorphism we have $\ker \phi \subset \text{Soc}(K_2)\subset m_xK_2$ and so  $K_2/m_xK_2\xrightarrow{\sim} K_1/m_xK_1$ which concludes the proof that the singularity satisfies $(\star_1).$

\end{remark}

\begin{remark}
The statement that one can always deform to a CY threefold with only nodes holds more generally without the crepant assumption and in the Du Bois setting (see Corollary \ref{cor: 3foldCYnotODP}).
\end{remark}

\section{Smoothing Fano varieties}
We begin by proving unobstructedness results for Fano varieties (cf. \cite[Theorem 4.5]{friedman2025deformations} and \cite[Theorem 3.17]{friedman2025unobstructed}). 

\begin{theorem}\label{thm: FanoUnobstructed}
Let $Y$ be a Fano variety with isolated lci singularities. Then $\TT^i_Y=0$ for $i\geq 2$ and, in particular, deformations of $Y$ are unobstructed.
\end{theorem}

\begin{proof}
We have $\TT^i_Y=\Ext^i(\Omega_Y^1,\mathcal{O}_Y)$ and so, by Serre duality, it suffices to show $H^{n-i}(\Omega_Y^1\otimes \omega_Y)=0.$ Consider the following hypercohomology spectral sequence: 
$$
E_2^{p,q}=H^p(\H^q \DB_Y^1\otimes \omega_Y) \Rightarrow \HH^{p+q}(\DB_Y^1\otimes \omega_Y)
.$$ Note that $\H^0(\DB_Y^1)\cong \Omega_Y^1$ since $\Omega_Y^1$ is reflexive and $\H^0(\DB_Y^1)$ is torsion-free. Therefore, we have $H^{n-i}(\Omega_Y^1\otimes \omega_Y)= E^{n-i,0}_2$ and by Theorem \ref{thm: duboisinv} we have $E^{0,n-i-1}_2=0.$ Hence, $H^{n-i}(\Omega_Y^1\otimes \omega_Y)= E^{n-i,0}_\infty\ito \HH^{n-i}(\DB_Y^1\otimes \omega_Y).$ But since $n-i+1<n$ and $\omega_Y^{-1}$ is ample it follows $\HH^{n-i}(\DB_Y^1\otimes \omega_Y)=0$ by a dual Nakano-type vanishing theorem for Du Bois complexes (see \cite[Theorem V.7.10]{guillen2006hyperresolutions} and \cite[Theorem C]{popashen}) which concludes the proof.
\end{proof}

\begin{remark}\label{rmk: 0limunobstr}
 We mention that the above unobstructedness result is new in the 0-liminal case (Du Bois but not rational). Note that our assumptions (Fano variety with isolated lci singularities) automatically imply the singular points are Du Bois by \cite[Corollary 7.6]{popa2024injectivity}.
\end{remark}

\begin{remark}
We learned from Robert Friedman that the proof of \cite[Theorem 3.17]{friedman2025unobstructed} also works in the weak Fano case (i.e., when $\omega_Y^{-1}$ is nef and big). In the weak Fano case, one should not expect $\TT^2_Y=0$ \cite[Example 2.8]{sano} and a technique such as $T^1$-lifting is needed.
\end{remark}

Theorem \ref{thm: FanoUnobstructed} holds more generally when we relax the lci assumption.  This will be replaced by a condition involving the invariant $\text{lcdef}(Y)$ introduced in \cite{popashen} which measures how far from lci the variety is. The condition $\text{lcdef}(Y)=0$ roughly says that $Y$ numerically behaves like an lci variety. 
For the definition and more details about this invariant we refer the reader to \cite{popashen}. In the following theorem and example, we use the term Fano without imposing the lci hypothesis of Definition \ref{def: CYvar}.

\begin{theorem}\label{thm: unobstrnonlci}
Let $Y$ be a Fano variety with isolated Gorenstein singularities. Assume that $\text{lcdef}(Y)=0$ and $\text{depth}(\H^0(\DB^1_Y))\geq n-1.$ Then $\TT^i_Y=0$ for $2\leq i< n-1$ and, in particular, if $n\geq 4$ deformations of $Y$ are unobstructed.
\end{theorem}

\begin{proof}
First, the natural map  $ \Ext^i(\Omega^1_Y,\mathcal{O}_Y)\to \Ext^i(\LL_Y,\mathcal{O}_Y)=\TT^i_Y$ is an isomorphism by a study of the Ext spectral sequence:
$$E_2^{p,q}=\Ext^p(\H^{-q}(\LL_Y),\mathcal{O}_Y) \Rightarrow \Ext^{p+q}(\LL_Y,\mathcal{O}_Y)$$ where we note the negative cohomology of the cotangent complex $\H^{-q}(\LL_Y)$ has 0-dimensional support for $q>0.$  Then, as in the previous theorem, we wish to show $H^{n-i}(\Omega_Y^1\otimes \omega_Y)=0.$ Note since $n-i\geq 2$ we have $H^{n-i}(\Omega_Y^1\otimes \omega_Y)=H^{n-i}(\Omega_Y^{[1]}\otimes \omega_Y)=H^{n-i}( \H^0(\DB_Y^1)\otimes \omega_Y).$ We would like to make use of a Nakano-type vanishing for singular varieties. Since $\text{lcdef}(Y)=0$ we have  $\HH^{n-i}(\DB_Y^1\otimes \omega_Y)=0$ by the dual Nakano-type vanishing of Popa-Shen \cite[Theorem C]{popashen}. Finally, to prove that $H^{n-i}(\Omega_Y^1\otimes \omega_Y)\ito \HH^{n-i}(\DB_Y^1\otimes \omega_Y)$ it would suffice to have $\H^{n-i-1}(\DB^1_Y)=0.$ This follows from a result of Popa-Shen-Vo \cite[Proposition F]{popa2024injectivity} which, in our case, says $\H^{k}(\DB^1_Y)=0$ for $0<k<m-1$ where $m=\min \{\text{depth}(\H^0(\DB^1_Y)), n- \text{lcdef}(Y)\}.$ By our assumptions we conclude $\H^{n-i-1}(\DB^1_Y)=0$ as $n-i-1\leq n-3.$
\end{proof}
\begin{remark}
It is immediate from definition that an lci variety has $\text{lcdef}=0.$ Note, moreover, that the condition $\text{depth}(\H^0(\DB^1_Y))\geq n-1$ holds in the lci case since $\depth (\Omega^1_Y) \geq n-1$ by Greuel \cite[Lemma 1.8]{greuel} and because $\Omega^1_Y$ is reflexive we also have $\H^0(\DB^1_Y)\cong \Omega^1_Y.$ 
\end{remark}

\begin{example}
We give non-lci examples to which the previous theorem applies. Let $Y$ be any Fano variety of dimension $\geq 4$ with Gorenstein isolated quotient singularities. For example, simplicial toric varieties have quotient singularities and, if they have $\dim \geq 4$ and isolated singularities, they cannot be lci (see Remark \ref{remark: rigidFano}). Let us now explain why $Y$ satisfies the assumptions of the theorem. As $Y$ has quotient singularities we have $\text{lcdef}(Y)=0$ by \cite[Corollary 11.22]{MP-Hodgefiltr}. The characterization of lcdef in terms of the Du Bois complex \cite[Corollary 2.3]{popashen}, then implies $\depth (\DB _Y^1)\geq n-1.$ Now quotient singularities are pre-1-Du Bois \cite[Proposition E]{shen2023k} which means $\DB_Y^1\simeq \H^0(\DB^1_Y)$, and, therefore, we conclude the condition $\text{depth}(\H^0(\DB^1_Y))\geq n-1$ in Theorem \ref{thm: unobstrnonlci} is also satisfied. 
\end{example}
We now move towards studying when singular Fano varieties admit smoothings.

\vspace{0.5em}

\noindent\textbf{Notation.} Let $Y$ be a singular Fano variety of dimension $n\geq 3$ which has isolated singularities. Let $\pi: \Yhat \to Y$ be a strong log resolution with exceptional divisor $E.$ If $x$ is a singularity, let $(X_x, x)$ denote its germ, where $X_x$ is a contractible Stein representative.
By restriction we have a resolution $\pi: \Xhat_x \to X_x$ with exceptional divisor $E_x = \pi^{-1}(x).$ Let $U_x:=X_x\setminus \{x\},$ $Z:= \text{Sing}(Y)$ and $V := Y \setminus Z.$

\begin{theorem}\label{thm: fanothm}
Suppose $Y$ is a singular Fano variety with isolated Du Bois lci singularities. 
\begin{enumerate}
    \item Then $Y$ can be deformed to a Fano variety whose singularities are 1-rational.
    \item Assume $Y$ has only non-1-rational hypersurface singularities, or, more generally, that every singularity of $Y$ satisfies $(\star_2)$ (see Definition \ref{def: star}). Then $Y$ is smoothable, and every small smoothing is Fano. 
\end{enumerate}
\end{theorem}
\begin{proof}

 First, recall that by a result of Schlessinger $$H^0(T^1_Y)\cong H^2_Z(Y, T^0_Y)\cong \bigoplus\limits_{x \in \mathrm{Sing}(Y)}  H^1(T_{U_x}),$$ and a global version also identifies $\TT^1_Y$ with $H^1(T_V)$ (see \cite[Lemma 4.1]{friedman2025deformations}). Moreover, one has a commutative diagram

\[\begin{tikzcd}
	{H^1(V,T_V)} & {H^2_Z(T^0_Y)}\cong \bigoplus\limits_{x \in \mathrm{Sing}(Y)}  H^1(T_{U_x}) \\
	{\TT^1_Y} & {H^0(T^1_Y)}
	\arrow[from=1-1, to=1-2]
	\arrow["\cong", from=2-1, to=1-1]
	\arrow[from=2-1, to=2-2]
	\arrow["\cong", from=2-2, to=1-2]
\end{tikzcd}\]

Let $\G:=\Omega_{\Yhat}^{n-1}(\log E)(-E)\otimes \pi^*\omega_Y^{-1}.$ After trivializing $\omega_Y$ in a neighborhood of the singular
points we have the following commutative diagram:  

\[\begin{tikzcd}[column sep=0.8em, cells={nodes={scale=0.9}}]
	{\TT^1_Y \cong H^1(V,T_V)} & {H^2_E(\G)} & {H^2(\G)} \\
	H^0(T^1_Y)\cong {\bigoplus\limits_{x \in \mathrm{Sing}(Y)} H^1(U_x, T_{U_x})} & { \bigoplus\limits_{x \in \mathrm{Sing}(Y)} H^2_{E_x}(\G)\cong\bigoplus\limits_{x \in \mathrm{Sing}(Y)} H^2_{E_x}(\Omega_{\Xhat}^{n-1}(\log E_x)(-E_x))}
	\arrow["h", from=1-1, to=1-2]
	\arrow[from=1-1, to=2-1]
	\arrow["{h'}", from=1-2, to=1-3]
	\arrow[from=2-1, to=2-2]
	\arrow[equals, from=2-2, to=1-2]
\end{tikzcd}\]

The top row comes from local cohomology sequences on  $\Yhat$ by noting $T_V\cong  \Omega_{V}^{n-1}\otimes \omega_V^{-1}\cong\G|_V.$ The second vertical map is an isomorphism by excision.

Now note that $H^2(\G)=H^2(\Omega_{\Yhat}^{n-1}(\log E)(-E)\otimes \pi^*\omega_Y^{-1})=\HH^{2}(\underline\Omega_Y^{n-1}\otimes \omega_Y^{-1})=0$ by a Nakano-type vanishing for Du Bois complexes \cite[Theorem V.5.1]{guillen2006hyperresolutions} (cf. also \cite[Proof of Prop 7.30]{peters2008mixed}). Therefore, $h$ is surjective. 

Let $(Y,x)$ be an isolated Du Bois lci singularity and assume it is not 1-rational. We claim the image of $\TT^1_Y$ in $T^1_{Y,x}\cong H^1(U_x,T_{U_x})$ is not contained in $m_xT^1_{Y,x}.$ Otherwise, we would have the image of $\TT^1_Y$ in $ H^2_{E_x}(\G)$ is contained in $m_xH^2_{E_x}(\G).$ But this cannot happen since $h$ is surjective and $H^2_{E_x}(\G)\cong  H^2_{E_x}(\Omega_{\Xhat_x}^{n-1}(\log E_x)(-E_x))$ is nonzero by the 1-irrational assumption. Hence, the image of $\TT^1_Y$ in $T^1_{Y,x}$ is not contained in $m_xT^1_{Y,x}.$

We begin by showing (2). We claim that under our assumption for any singularity $(Y,x)$ one has $\TT^1_Y \twoheadrightarrow T^1_{Y,x}/m_xT^1_{Y,x}.$ If $(Y,x)$ is a hypersurface singularity $T^1_{Y,x}$ is a cyclic $\mathcal{O}_{Y,x}$-module. Hence, if the singularity is non-1-rational by the argument above we have the image of $\TT^1_Y$ in $T^1_{Y,x}$ is not contained in $m_xT^1_{Y,x},$  and, therefore, $\TT^1_Y \twoheadrightarrow T^1_{Y,x}/m_xT^1_{Y,x}.$ 

If the singularity more generally satisfies $(\star_2)$ then $h$ being surjective implies that $\TT^1_Y\to K_2/m_xK_2$ is a surjection, where $K_2=H^2_{E_x}(\G)\cong H^2_{E_x}(\Omega_{\Xhat_x}^{n-1}(\log E_x)(-E_x))$. So by $(\star_2)$ we again have $\TT^1_Y\twoheadrightarrow T^1_{Y,x}/m_xT^1_{Y,x}\xrightarrow{\sim} K_2/m_xK_2.$

By Theorem \ref{thm: FanoUnobstructed}, $\text{Def}(Y)$ is smooth. Hence we may apply Lemma \ref{lemma: smoothinglci} with $S'=\text{Def}(Y)$ and $T=\TT^1_Y$, and from the fact that $\TT^1_Y \twoheadrightarrow T^1_{Y,x}/m_xT^1_{Y,x}$ we conclude that $Y$ is smoothable.

We now show (1) i.e. that $Y$ can always be deformed to a Fano variety whose singularities are 1-rational. Note that by the semicontinuity properties of the minimal exponent \cite[Theorem 1.2]{vfiltr}, the 1-rational singularities will remain 1-rational after a small deformation. So, if $Y$ had only hypersurface singularities the conclusion would follow immediately since we showed above one can smooth out the non-1-rational hypersurface singularities.

To handle the lci case as well we make use of Lemma $\ref{lemma: deforminglci}$ which roughly tells us that a general deformation either smooths $(Y,x)$ or splits it into at least two singularities. The idea will be to iterate this process until all the singularities are 1-rational. 

So we begin with $Y$ and recall we showed before if $(Y,x)$ is not 1-rational then the image of $\TT^1_Y$ in $T^1_{Y,x}$ is not contained in $m_xT^1_{Y,x}.$ Since $\text{Def}(Y)$ is smooth by Theorem \ref{thm: FanoUnobstructed}, we may apply Lemma \ref{lemma: deforminglci} with $S'=\text{Def}(Y)$ and $T=\TT^1_Y.$ Then a sufficiently small general deformation of $Y$ either smooths $x$ or splits it into singular points $x_1,\ldots,x_r$ with $r\geq 2.$ We choose a deformation $Y_{1}$ for which this holds simultaneously at every non-1-rational singular point of $Y.$

If $x$ splits, by conservation of the Milnor number for isolated lci singularities (see, for example, \cite[Theorem 4.2]{MR3830255}) we have
$$\sum_{j=1}^r \mu(Y_{1},x_j)\leq \mu(Y,x).$$
Since each $x_j$ is singular, its Milnor number is at least one. Hence, as $r\geq 2$, for each $j$ we obtain
$\mu(Y_{1},x_j)<\mu(Y,x).$

Again by the semicontinuity of the minimal exponent \cite[Theorem 1.2]{vfiltr}, the 1-rational singularities will remain 1-rational.  Hence, we obtain a new Fano variety $Y_1$ with isolated lci Du Bois singularities and if non-1-rational singularities are present one may repeat the process.
This process terminates, since at each step the maximum Milnor number among the non-1-rational singularities strictly decreases.

By openness of versality \cite[pp.~600-601]{MR382729}, every sufficiently small deformation of a nearby fiber of the semiuniversal family of $Y$ is again a fiber of this family. Thus, choosing each deformation sufficiently small, all the steps above take place in the same family, and the resulting variety is a deformation of $Y$ which has only $1$-rational singularities. 
In fact, the argument shows that the parameters corresponding to fibers with only 1-rational singularities form an open dense subset of the semiuniversal deformation space.

\end{proof}
\begin{remark}\label{remark: rigidFano}
It is known by a result of Totaro \cite[Theorem 5.1]{Totaro} that a toric Fano variety, which is smooth in codimension 2
and $\QQ$-factorial in codimension 3 is rigid. Let us briefly explain why this does not contradict Theorem \ref{thm: fanothm}.

First, we claim that if $Y$ is a singular toric threefold with isolated singularities that is $\QQ$-factorial in codimension 3 (so, $\QQ$-factorial) then it is not lci. Since $Y$ is simplicial
it has pre-1-rational singularities \cite[Proposition E]{shen2023k}. Hence, if it were lci it would have 1-rational singularities. But threefolds with 1-rational singularities must be smooth.

If $Y$ is a singular toric variety of dimension $\geq 4$ with isolated singularities, then we claim that again it is not lci. Suppose, for contradiction, that $Y$ is lci. By a result of Grothendieck \cite[Corollary 3.14]{gro} (also known as Samuel's conjecture), since $Y$ is factorial in codimension 3, it is in fact factorial. Therefore, $Y$ is a simplicial toric variety with isolated singularities, so its singularities are 1-rational \cite[Proposition E]{shen2023k}. But lci toroidal singularities cannot be 1-rational \cite[Corollary 1.6]{QBM}. We conclude that one cannot apply Totaro's result in the lci isolated case.

\end{remark}

\section{Smoothing Calabi-Yau varieties}

\paragraph{\textbf{Notation}} Let $Y$ be a singular Calabi-Yau variety which has isolated Du Bois lci singularities. Let $\pi: \Yhat \to Y$ be a strong log resolution with exceptional divisor $E.$ If $x$ is a singularity, let $(X_x, x)$ denote its germ, where $X_x$ is a contractible Stein representative. By restriction we have a resolution $\pi: \Xhat_x \to X_x$ with exceptional divisor $E_x = \pi^{-1}(x).$ Let $U_x:=X_x\setminus \{x\}$ and $V := Y \setminus \text{Sing}(Y).$

 \begin{remark}\label{remark: analyticsp}
While, for simplicity, we will assume the Calabi-Yau variety $Y$ is a compact K\"ahler variety throughout, as noted in Definition \ref{def: CYvar}, to obtain the smoothing results (i.e. part (2) of the main theorems), one could also take more generally $Y$ to be a compact complex analytic variety with isolated singularities that admits a resolution satisfying the $\partial\overline{\partial}$-lemma as in \cite{friedman2025deformations}. In this case, one can construct an analogue of the Du Bois complex with Hodge spectral sequence that degenerates at $E_1$ (see \cite[Remark 4.3.]{friedman2024higher}).
\end{remark}

\begin{remark}\label{rmk: projsm}
When $Y$ is Calabi-Yau and projective, to get projective smoothings one should also assume, for instance, that $H^2(\mathcal{O}_Y)=0,$ although this is not necessary if one only needs analytic smoothings.
\end{remark}

To produce smoothings we need the following unobstructedness result due to Friedman:
\begin{theorem}\textup{(\cite[Theorem 1.6]{friedman2025unobstructed})}
Let $Y$ be a singular Calabi-Yau variety of dimension $n \ge 3$, such that the singularities of $Y$ are isolated Du Bois and lci. Then the deformations of $Y$ are unobstructed.
\end{theorem}

\begin{remark}
 Although we will not make use of this here, we note that one can extend Friedman's result outside of the lci case as in our unobstructedness result in the Fano case (Theorem \ref{thm: unobstrnonlci}). Indeed, if $n\geq 4$ one can replace the lci assumption by the condition $\text{lcdef}(Y)=0$ and $\text{depth}(\H^0(\DB^1_Y))\geq n-1.$ The idea is to proceed via $T^1$ lifting as in the Proof of Theorem 1.6 of \cite[page 13]{friedman2025unobstructed}. One can check that the conclusion of Lemma 2.6 in \cite{friedman2025unobstructed} still holds under our assumptions as $\H^{n-3}(\DB^1_Y)=0$ by \cite[Proposition F]{popa2024injectivity}.
\end{remark}

The following proposition will allow us to lift appropriate first-order local deformations to global deformations. 
\begin{proposition}\label{prop:0map}
Suppose $Y$ is a compact complex analytic variety of dimension $n\geq 3$ with isolated lci Du Bois singularities. Assume that $Y$ is K\"ahler (or, more generally, that it admits a resolution satisfying the $\partial\overline{\partial}$-lemma). Then the morphism $$H^2_{E}(\Omega^{n-1}_{\Yhat}(\log E))\cong 
\bigoplus\limits_{x \in \mathrm{Sing}(Y)} H^2_{E_x}(\Omega^{n-1}_{\Xhat_x}(\log E_x)) \longrightarrow H^2(\widehat{Y}, \Omega^{n-1}_{\Yhat}(\log E))
$$ is zero.
\end{proposition}

\begin{proof}

Note that by Serre duality $H^2(\widehat{Y}, \Omega^{n-1}_{\Yhat}(\log E))$ is dual to 
$H^{n-2}(\Omega_{\Yhat}^1(\log E)(-E))$ and by local duality $H^2_{E_x}(\Omega^{n-1}_{\Xhat_x}(\log E_x))$ is dual to $H^{n-2}(\Omega^{1}_{\Xhat_x}(\log E_x )(-E_x)).$

Therefore, dually, the statement of the proposition is equivalent to showing that 
$$
H^{n-2}(\Omega_{\Yhat}^1(\log E)(-E)) \rightarrow 
\bigoplus\limits_{x \in \mathrm{Sing}(Y)} H^{n-2}(\Omega^{1}_{\Xhat_x}(\log E_x )(-E_x))
$$
is zero. 

Consider the following hypercohomology spectral sequence 
$$
E_2^{p,q}=H^p(\H^q \DB_Y^1) \Rightarrow \HH^{p+q}(\DB_Y^1)
.$$ Since the singularities of $Y$ are isolated, we have 
$\DB_Y^1 \simeq \RR \pi_* \Omega_{\Yhat}^1(\log E)(-E),$ and so, 
the spectral sequence becomes 
$$
E_2^{p,q}=H^p(R^q \pi_{*} \Omega_{\Yhat}^1 (\log E)(-E)) 
\Rightarrow H^{p+q}
(\Omega_{\Yhat}^1 (\log E)(-E)).$$
For $q>0$, the sheaves $R^q\pi_*\Omega^1_{\widehat{Y}}(\log E)(-E)$ are supported on finitely many points, hence $H^p(R^q\pi_*\Omega^1_{\widehat{Y}}(\log E)(-E))=0$ for $p>0$. Therefore $E_2^{p,q}=0$ whenever both $p,q>0.$ Note, $R^{n-2}\pi_*\Omega^1_{\widehat{Y}}(\log E)(-E)$ is a skyscraper sheaf with stalk at $x$ equal to $H^{n-2}(\Omega^1_{\widehat{X}}(\log E_x)(-E_x)).$ Hence,
$$H^0(R^{n-2}\pi_*\Omega^1_{\widehat{Y}}(\log E)(-E))\cong \bigoplus_{x\in \mathrm{Sing}(Y)} H^{n-2}(\Omega^1_{\Xhat_x}(\log E_x)(-E_x)).$$

We would like to show 
\begin{align*}
H^{n-2}(\Omega_{\Yhat}^1(\log E)(-E))
\rightarrow E_2^{0,n-2}
&= H^0(R^{n-2}\pi_*\Omega^1_{\Yhat}(\log E)(-E)) \\
&= \bigoplus\limits_{x \in \mathrm{Sing}(Y)} H^{n-2}(\Omega^{1}_{\Xhat_x}(\log E_x)(-E_x))
\end{align*}
is zero. In other words, we want to show $E_\infty^{0,n-2}=0.$ 

Note $E_2^{n-2,0}=E_\infty^{n-2,0}$ which is immediate for $n=3$ and for $n>3$ follows from
$$ E_2^{0,n-3}= \bigoplus\limits_{x \in \mathrm{Sing}(Y)} H^{n-3}(\Omega^{1}_{\Xhat_x}(\log E_x )(-E_x)) =0 $$
by \cite[Theorem 5]{steenbrink1997bois}, as the singularities are isolated lci. Hence, we obtain an injection $E_2^{n-2,0} =E_\infty^{n-2,0} \ito  \HH^{n-2}(\DB_Y^1).$

However, by \cite[Proposition 5.2]{popa2024injectivity} since the singularities are Du Bois one has the map $E_\infty^{n-2,0}=E_2^{n-2,0}= H^{n-2}(\H^0(\DB_Y^1))\rightarrow  \HH^{n-2}(\DB_Y^1)$ is also surjective, and so, it is an isomorphism (cf. also \cite[Remark 2.5]{friedman2025unobstructed}). Note \cite[Proposition 5.2]{popa2024injectivity} makes use of the $E_1$-degeneration of the Hodge-to-de Rham spectral sequence which is where we use the K\"ahler hypothesis. We conclude $E_{\infty}^{0,n-2}=0$ and so the morphism we were interested in is zero.
\end{proof}

The first result concerns smoothing a Calabi-Yau variety whose singularities are not 1-Du Bois. In this case, no global condition is needed.

\begin{theorem}\label{thm: lcinot1DB}
Suppose $Y$ is a singular Calabi-Yau variety with isolated Du Bois lci singularities. 
\begin{enumerate}
    \item Then $Y$ can be deformed to a Calabi-Yau variety whose singularities are 1-Du Bois.
    \item Assume $Y$ has only non-1-Du~Bois hypersurface singularities, or, more generally, that every singularity of $Y$ satisfies $(\star_1).$ Then $Y$ is smoothable, and every small deformation is Calabi-Yau.
\end{enumerate} 
\end{theorem}
\begin{proof}

Recall one can identify $T^1_{Y,x}$ with $H^1(U_x, \Omega^{n-1}_{U_x})$ and, moreover, a global version of this says $\TT^1_Y \cong H^1(V,\Omega_{V}^{n-1})$ (see \cite[Lemma 4.1]{friedman2025deformations}).

Consider the following commutative diagram:
\[\begin{tikzcd}
	{\TT^1=H^1(V,\Omega_{V}^{n-1})} & {H^2_E(\Omega_{\Yhat}^{n-1}(\log E))} & {H^2(\Omega_{\Yhat}^{n-1}(\log E))} & {} \\
	{\TT^1=H^1(V,\Omega_{V}^{n-1})} & {H^2_E(\Omega_{\Yhat}^{n-1})} & {H^2(\Omega_{\Yhat}^{n-1})} \\
	{H^0(T^1_Y)\cong H^2_Z(T^0_Y)\cong\bigoplus H^1(U_x, \Omega^{n-1}_{U_x})}
	\arrow["f"', from=1-1, to=1-2]
	\arrow[equals, from=1-1, to=2-1]
	\arrow["{f'}"', from=1-2, to=1-3]
	\arrow["g"{pos=0.7}, from=2-1, to=2-2]
	\arrow["\eta", from=2-1, to=3-1]
	\arrow["\psi", from=2-2, to=1-2]
	\arrow["{g'}", from=2-2, to=2-3]
	\arrow[from=2-3, to=1-3]
	\arrow["\varphi"{pos=0.7}, from=3-1, to=1-2]
	\arrow["\phi"'{pos=0.6}, from=3-1, to=2-2]
\end{tikzcd}\]
First, by excision $$H^2_E(\Omega_{\Yhat}^{n-1}(\log E))\cong \bigoplus\limits_{x \in \mathrm{Sing}(Y)} H^2_{E_x} (\Omega_{\Xhat_x}^{n-1}(\log E_x)).$$ Hence, the morphism $H^0(T^1_Y)\rightarrow H^2_E(\Omega_{\Yhat}^{n-1}(\log E))$ is given by $\bigoplus\limits_{x \in \mathrm{Sing}(Y)} \varphi_x$ where $\varphi_x$ is the composition  $$H^1(U_x, \Omega^{n-1}_{U_x})\xto{\phi_x} H^2_{E_x} (\Omega_{\Xhat_x}^{n-1}) \xto{\psi_x} H^2_{E_x} (\Omega_{\Xhat_x}^{n-1}(\log E_x)).$$ By Section \ref{sec: 2.3} we have the following short exact sequence $$0\rightarrow H^1 (\Omega_{\Xhat_x}^{n-1}(\log E_x)) \rightarrow H^1(U_x, \Omega^{n-1}_{U_x}) \xto{\varphi_x} H^2_{E_x} (\Omega_{\Xhat_x}^{n-1}(\log E_x)) \rightarrow 0.$$  
 
For each singularity $(Y,x)$ that is not 1-Du Bois we obtain that each $H^2_{E_x} (\Omega_{\Xhat_x}^{n-1}(\log E_x))$ is nonzero. Consider $\beta'_x\in H^2_{E_x} (\Omega_{\Xhat_x}^{n-1}(\log E_x))$ but such that $\beta'_x\notin m_xH^2_{E_x} (\Omega_{\Xhat_x}^{n-1}(\log E_x)).$ 
Now, by Proposition \ref{prop:0map} the map $H^2_E(\Omega_{\Yhat}^{n-1}(\log E)) \rightarrow H^2(\Omega_{\Yhat}^{n-1}(\log E))$ is zero. Hence, the element $\beta'=\bigoplus\limits_{x \in \mathrm{Sing}(Y)} \beta'_x$ comes from a global first-order deformation and, therefore, 
$$\text{im}(H^1(\Omega_{V}^{n-1}) \to H_{E_x}^2(\Omega_{\Xhat_x}^{n-1}(\log E_x))) \not\subset m_x H_{E_x}^2(\Omega_{\Xhat_x}^{n-1}(\log E_x))
.$$ But then looking at the composition $$\TT^1_Y= H^1(\Omega_{V}^{n-1}) \to T^1_{Y,x}\sto H_{E_x}^2(\Omega_{\Xhat_x}^{n-1}(\log E_x))$$ we deduce the image of $\TT^1_Y$ in $T^1_{Y,x}$ is not contained in $m_xT^1_{Y,x}$ when $(Y,x)$ is not 1-Du Bois.

We begin by showing (2). We claim that under our assumption for any singularity $(Y,x)$ one has $\TT^1_Y \twoheadrightarrow T^1_{Y,x}/m_xT^1_{Y,x}.$ If $(Y,x)$ is a hypersurface singularity $T^1_{Y,x}$ is a cyclic $\mathcal{O}_{Y,x}$-module. Hence, if the singularity is non-1-Du Bois by the argument above we have the image of $\TT^1_Y$ in $T^1_{Y,x}$ is not contained in $m_xT^1_{Y,x},$  and, therefore, $\TT^1_Y \twoheadrightarrow T^1_{Y,x}/m_xT^1_{Y,x}.$

If the singularity more generally satisfies $(\star_1)$ then $f$ being surjective implies that $\TT^1_Y\to K_1/m_xK_1$ is a surjection, where $K_1=H^2_{E_x}(\Omega_{\Xhat_x}^{n-1}(\log E_x))$. So by $(\star_1)$ we again have $\TT^1_Y\twoheadrightarrow T^1_{Y,x}/m_xT^1_{Y,x}\xrightarrow{\sim} K_1/m_xK_1.$

By Friedman's unobstructedness Theorem \ref{thm: unobstructed}, $\text{Def}(Y)$ is smooth. Hence we may apply Lemma \ref{lemma: smoothinglci} with $S'=\text{Def}(Y)$ and $T=\TT^1_Y$, and from the fact that $\TT^1_Y \twoheadrightarrow T^1_{Y,x}/m_xT^1_{Y,x}$ we conclude that $Y$ is smoothable.

We now show (1) i.e. that $Y$ can always be deformed to a CY variety whose singularities are 1-Du Bois. The proof proceeds by iteration exactly as in the proof of Theorem \ref{thm: fanothm}. The only thing left to check is that after deforming one can keep applying the above procedure to be able to deform again. Note that after a small deformation of $Y$ to $Y_1$ we still have that $Y_1$ is a compact complex analytic space with isolated lci Du Bois singularities and also it is Calabi-Yau since $\omega_{Y_1}\cong \mathcal{O}_{Y_1}$ (see for instance \cite[proof of Corollary 1.5]{friedman2024higher})  and also $H^1(\mathcal{O}_{Y_1})=0$ by upper semicontinuity. We now show $Y_1$ is K\"ahler which is enough to apply the procedure again. The openness of the K\"ahler condition follows from Bingener's criterion \cite[Theorem 6.3]{MR701367} which says that a K\"ahler central fiber $Y$ such that $H^2(Y,\mathbb R)\to H^2(Y, \mathcal O_Y)$ is surjective remains K\"ahler under small deformations. We briefly explain why the criterion applies here (cf. also \cite[Proposition 5]{MR1863856} and \cite[Lemma 4.2]{Imagi2026}). To see that $H^2(Y,\mathbb R)\to H^2(Y, \mathcal O_Y)$ is surjective we first claim $H^2(Y)$ is pure. Indeed, the long exact sequence for the pair $(\Yhat, E)$ gives 
$$H^1(\Yhat)\to H^1(E)\to H^2(\Yhat,E)=H^2(Y)\to H^2(\Yhat)$$ 
and suppose for the moment that $H^1(E)=0$ then  $H^2(Y)\xhookrightarrow{} H^2(\Yhat)$ so by strictness $H^2(Y)$ is pure. To see $H^1(E)=0$ consider the long exact sequence 
$$H^1_{E_x}(\Xhat)\to H^1(E_x)\to H^1(L_x)$$
where $L_x$ is the link of the singularity $(Y,x).$ Then $H^1_{E_x}(\Xhat)=0$ since $H^{2n-1}(E,\mathbb C)=0$ and $H^1(L)=0$ since the singularity is isolated lci. Therefore, $H^1(E)=0.$

Since $H^2(Y)$ is pure of weight $2$, it has a Hodge decomposition $H^2(Y,\mathbb C)=H^{2,0}\oplus H^{1,1}\oplus H^{0,2}$ with $H^{0,2}=\overline{H^{2,0}}.$
Finally, since $Y$ has Du Bois singularities $H^2(\mathcal O_Y)= H^{0,2}(Y)$ so 
$H^2(Y,\mathbb C)\to H^{0,2}(Y)=H^2(Y, \mathcal O_Y)$ is the natural projection map.
Hence, $H^2(Y,\mathbb R)\to H^{0,2}(Y)=H^2(Y, \mathcal O_Y)$ is surjective since $\alpha+\overline{\alpha} $ maps to $\alpha$ for any $\alpha \in H^{0,2}(Y).$

\end{proof}

The following example explains why an additional local condition such as $(\star_1)$ is needed for the smoothing results in the lci case. In other words, assuming the singularities are not 1-Du Bois is not sufficient outside the hypersurface case.  This example was constructed using AI models. 
\begin{example}\label{ex: notsmoothableglobal}

Consider the weighted projective space $\mathbb{P}=\mathbb{P}(1^{19},3,4)$ with coordinates $x_1,\ldots,x_{19},y,z$, and set
$$q(x)=\sum_{i=1}^{19}x_i^2,\text{  }G(x)=\sum_{i=1}^{19}3^{-11i}x_i^{24},\text{ and } H=(y^4-z^3)^2+G(x)$$ 
Consider the complete intersection $Y=\{q=0, H=0\} \subset\mathbb{P}.$
We will show that $Y$ is a nonsmoothable Calabi-Yau variety with a unique isolated rational lci singularity which is also not 1-Du Bois.

The idea of the construction is that the weights force every deformation of $q$ to remain singular along $C=\{x_i=0\}$, as in \cite[proof of Lemma 18.14]{MR1798982} and \cite[Section 1.2, equation (1.1)]{MR4399115}.
Moreover, the second hypersurface continues to meet $C$ where the differential of the first equation vanishes, forcing the complete intersection to be singular. Thus every small deformation of the equations is singular. The last thing to check will be that every small abstract deformation of $Y$ is obtained by deforming the equations.

We first determine the singular locus of $Y$. Note that $Y$ avoids the two singular coordinate points of $\mathbb{P}$, so it lies in the smooth open subset $W=\mathbb{P}_{\mathrm{reg}}$. Since $q$ is irreducible and $H\notin(q)$, the equations form a regular sequence. Thus $Y$ is lci of dimension $18$.

We claim that $dq$ and $dG$ are independent whenever $q(x)=0$ and $x\neq0$. Indeed, if $dG=\lambda\,dq$, then $x_i^{22}=(\lambda/12)3^{11i}$ for each nonzero $x_i$. Here $\lambda\neq0$, so choosing $\gamma$ with $\gamma^{11}=\lambda/12$, we can write $x_i^2=\gamma3^i\zeta_i$, where $\zeta_i^{11}=1$. If $m$ is the largest index with $x_m\neq0$, then
$$\left|\sum_{x_i\neq0}3^i\zeta_i\right|\geq3^m-\sum_{i=1}^{m-1}3^i>0,$$
contradicting $q(x)=0$. Since the $x$-component of $dH$ is $dG$, the Jacobian criterion shows that $Y$ is smooth wherever $x\neq0$.

It remains to consider the curve $C=\{x=0\}\cong\mathbb{P}(3,4)$. The intersection $Y\cap C$ is given by $(y^4-z^3)^2=0$ and is supported at $p=[0:\cdots:0:1:1]$. Both differentials vanish there, so $p$ is the unique singular point of $Y$. To check the singularity type, look in the chart $z=1$ near $p$ and consider local coordinates $u=y^4-1$. Then
$$(Y,p)\cong\{q(x)=0,\ u^2+G(x)=0\}\subset(\CC^{20},0).$$
This is an isolated weighted homogeneous complete intersection with weights $(1^{19},12)$ and degrees $2$ and $24.$ The difference between the sum of the weights and the sum of the degrees is $5$. By \cite[Corollary D]{MR5078859},
$$2<\widetilde{\alpha}_p(Y)\leq2+\frac{5}{24}<3,$$
where the first inequality is strict because $24\nmid5$. Thus $(Y,p)$ is rational, hence Du Bois, but not 1-Du Bois.

We next check that $Y$ is Calabi-Yau. Let $S=\CC[x_1,\ldots,x_{19},y,z]$ and $R=S/(q,H)$. Since the equations form a regular sequence, $R$ is Cohen-Macaulay of dimension $19$, and we have the twisted Koszul resolution $0\to\mathcal{O}_{\mathbb{P}}(k-26)\to\mathcal{O}_{\mathbb{P}}(k-2)\oplus\mathcal{O}_{\mathbb{P}}(k-24)\to\mathcal{O}_{\mathbb{P}}(k)\to\mathcal{O}_Y(k)\to0$. By \cite[Theorem 1.4.1]{MR704986}, we have $H^0(\mathbb{P},\mathcal{O}_{\mathbb{P}}(\ell))=S_\ell$ and $H^j(\mathbb{P},\mathcal{O}_{\mathbb{P}}(\ell))=0$ for $1\leq j\leq19$ and every $\ell\in\ZZ$. Splitting the Koszul resolution into two short exact sequences and taking cohomology therefore gives $H^i(Y,\mathcal{O}_Y(k))=0$ for $1\leq i\leq17$ and $k\in\ZZ$, and $H^0(Y,\mathcal{O}_Y(k))\cong R_k$ (see \cite[Lemma 6.4]{MR4709112} or \cite[Proposition 3B.4]{MR879909}). 

In particular, $H^0(Y,\mathcal{O}_Y)=\CC$, so $Y$ is connected. Being lci, $Y$ satisfies Serre's condition $(S_2)$, and being smooth away from $p$, it is regular in codimension one. Hence $Y$ is normal and reduced. Since $Y\subset W$ and $W$ is smooth, adjunction in $W$ gives $\omega_Y\cong\mathcal{O}_Y(2+24-19-3-4)\cong\mathcal{O}_Y$ (see \cite[Corollary 6B.9]{MR879909}). Together with $H^1(\mathcal{O}_Y)=0$, this shows that $Y$ is Calabi-Yau.

To prove nonsmoothability, we first show that every small abstract deformation of $Y$ is obtained by deforming the equations. The restricted weighted Euler sequence
$$0\to\mathcal{O}_Y\to\mathcal{O}_Y(1)^{\oplus19}\oplus\mathcal{O}_Y(3)\oplus\mathcal{O}_Y(4)\to T_W|_Y\to0$$
and the above vanishings imply $H^1(T_W|_Y)=0$. Hence the deformation sequence for $Y\subset W$ gives a surjection $H^0(N_{Y/W})\twoheadrightarrow\TT^1_Y$. Since $N_{Y/W}\cong\mathcal{O}_Y(2)\oplus\mathcal{O}_Y(24)$ and $H^0(Y,\mathcal{O}_Y(d))=R_d$, we obtain surjections
$$S_2\oplus S_{24}\twoheadrightarrow H^0(N_{Y/W})\twoheadrightarrow\TT^1_Y.$$

Consider the coefficient family near $(q,H)$. After shrinking the coefficient space, the first equation $Q$ remains a nondegenerate quadratic form in the $x_i$, and the coefficients of $y^8$ and $z^6$ in the second equation $H'$ remain nonzero. Thus $Q,H'$ form a regular sequence, and their common zero locus avoids the singular locus of $\mathbb{P}$. The resulting projective family is flat by the relative regular-sequence criterion, and the above composition is its Kodaira-Spencer map.

By Friedman's unobstructedness Theorem \ref{thm: unobstructed}, $\text{Def}(Y)$ is smooth. Since the Kodaira-Spencer map is surjective, the classifying map from the coefficient space to $\text{Def}(Y)$ is a submersion and admits a local analytic section. Hence every sufficiently small abstract deformation of $Y$ is obtained by deforming the equations.

Finally, we show that every nearby coefficient fiber $Y'=\{Q=H'=0\}$ is singular. Every weighted degree-two polynomial $Q$ depends only on the $x_i$, so $Q|_C=dQ|_C=0$. On the other hand, the restriction of the second equation has the form $H'|_C=Ay^8+By^4z^3+Dz^6$, with $(A,B,D)$ close to $(1,-2,1)$. On $C\cap\{yz\neq0\}$, using the coordinate $r=y^4/z^3$, this becomes $Ar^2+Br+D=0$. Its roots remain near $r=1$, counted with multiplicity, so $Y'$ meets $C$ near $p$. At each intersection point $dQ=0$, and the Jacobian of $(Q,H')$ has rank at most one. Thus $Y'$ is singular. Since every sufficiently small abstract deformation is obtained by deforming the equations, $Y$ is not smoothable.

\end{example}
\begin{remark}
This illustrates the splitting behavior in Lemma \ref{lemma: deforminglci}. A deformation direction outside $m_pT^1_{Y,p}$ can split the singularity into several singular points without smoothing it. Indeed, $Y_t=\{q=0,\ H-tz^6=0\}$ splits $p$ into two ordinary double points. As these are 1-rational we have no further control over smoothing them globally.
\end{remark}
\begin{remark}
The singularity in the example above does not satisfy $(\star_1)$. Indeed, its local equations have weights $(1^{19},12)$ and degrees $d_1=2$ and $d_2=24$, so
$$N=\sum_j a_j-\sum_i d_i=19+12-2-24=5.$$
By the weighted homogeneous criterion, condition $(\star_1)$ is equivalent to $N<d_1$, which fails here since $5>2$. Thus this example does not satisfy the additional local hypothesis in part (2) of Theorem \ref{thm: lcinot1DB}.
\end{remark}
\begin{remark}
Increasing the dimension by 1 in the construction above gives a nonsmoothable Fano variety $Y_F\subset\mathbb{P}(1^{20},3,4)$ of dimension $19$ with a unique isolated rational lci singularity which is not 1-Du Bois.
\end{remark}

Going back to proving smoothing results, we would now like to relax the not 1-Du Bois condition to not 1-rational. In other words, we would like to be able to smooth out 1-liminal points. As explained in Section \ref{sec: 1.3}, one should not expect a smoothing in the absence of additional global conditions. In the case of threefolds, Namikawa-Steenbrink show $\QQ$-factoriality is a global condition which, together with appropriate local conditions on the singularities, guarantees smoothability. 

In higher dimensions, we find that the appropriate global condition on $Y$ is the equality of Hodge-Du Bois numbers $\underline{h}^{n-1,2}=\underline{h}^{1,n-2},$ which for threefolds with rational singularities is equivalent to the familiar $\QQ$-factoriality condition (see \cite[Theorem E]{park2025q} and  \cite[Theorem 3.2]{namikawa1995global}).
\begin{theorem}\label{thm: lcigeneral}
Suppose $Y$ is a singular Calabi-Yau variety of dimension $n$ with isolated Du Bois lci singularities. Assume that we have the following equality of Hodge-Du Bois numbers $\underline{h}^{n-1,2}(Y)=\underline{h}^{1,n-2}(Y).$

\begin{enumerate}
    \item If $Y$ is projective with $H^2(\mathcal{O}_Y)=0$ then it can be deformed to a projective Calabi-Yau whose singularities are 1-rational.
    \item Assume $Y$ has only non-1-rational hypersurface singularities, or, more generally, that every singularity of $Y $satisfies $(\star_2).$ Then $Y$ is smoothable, and every small deformation is Calabi-Yau. 
\end{enumerate} 
\end{theorem}

\begin{remark}
In (1) one can work in the K\"ahler setting (without the assumption that $Y$ is projective with $H^2(\mathcal{O}_Y)=0$) if one assumes instead the equality of Betti numbers of singular cohomology $b_{n+1}(Y)=b_{n-1}(Y)$ (which in particular implies the equality of Hodge Du Bois numbers).

\end{remark}
\begin{proof}

Consider the following commutative diagram with exact rows coming from the local cohomology sequence:


\[\begin{tikzcd}[column sep=0.8em, cells={nodes={scale=0.9}}]
	{\TT^1=H^1(V,\Omega^{n-1}_V)} & {\bigoplus\limits_{x \in \mathrm{Sing}(Y)} H^2_{E_x}(\Omega_{\Xhat_x}^{n-1}(\log E_x))} & {H^2(\Omega_{\Yhat}^{n-1}(\log E))} & {} \\
	{\TT^1=H^1(V,\Omega^{n-1}_V)} & {\bigoplus\limits_{x \in \mathrm{Sing}(Y)} H^2_{E_x}(\Omega_{\Xhat_x}^{n-1})} & {H^2(\Omega_{\Yhat}^{n-1})} \\
	{\TT^1=H^1(V,\Omega^{n-1}_V)} & {\bigoplus\limits_{x \in \mathrm{Sing}(Y)} H^2_{E_x}(\Omega_{\Xhat_x}^{n-1}(\log E_x)(-E_x))} & {H^2(\Omega_{\Yhat}^{n-1}(\log E)(-E))} \\
	{H^0(T^1_Y)\cong \bigoplus\limits_{x \in \mathrm{Sing}(Y)} H^1(U_x, \Omega^{n-1}_{U_x})}
	\arrow["f", from=1-1, to=1-2]
	\arrow[equals, from=1-1, to=2-1]
	\arrow["{f'}", from=1-2, to=1-3]
	\arrow["g", from=2-1, to=2-2]
	\arrow[ equals, from=2-1, to=3-1]
	\arrow["{\psi_x}", from=2-2, to=1-2]
	\arrow["{g'}", from=2-2, to=2-3]
	\arrow[from=2-3, to=1-3]
	\arrow["h", from=3-1, to=3-2]
	\arrow["\eta"', from=3-1, to=4-1]
	\arrow["{\iota_x}", from=3-2, to=2-2]
	\arrow["{h'}", from=3-2, to=3-3]
	\arrow[from=3-3, to=2-3]
	\arrow["{s_x}"'{pos=0.6}, from=4-1, to=3-2]
\end{tikzcd}\]

We now show that $h'$ is zero which then allows us to lift an appropriate element in each component to obtain a global smoothing.

From Proposition \ref{prop:0map} the morphism $
\bigoplus\limits_{x \in \mathrm{Sing}(Y)} H^2_{E_x}(\Omega^{n-1}_{\Xhat_x}(\log E_x)) \xto{f'} H^2( \Omega^{n-1}_{\Yhat}(\log E))
$ is zero and, therefore, so is the composition $$\bigoplus\limits_{x \in \mathrm{Sing}(Y)} H^2_{E_x}(\Omega^{n-1}_{\Xhat_x}(\log E_x)(-E_x)) \to  \bigoplus\limits_{x \in \mathrm{Sing}(Y)} H^2_{E_x}(\Omega^{n-1}_{\Xhat_x}(\log E_x)) \rightarrow H^2( \Omega^{n-1}_{\Yhat}(\log E)).$$

\textbf{Claim.} The equality $\underline{h}^{n-1,2}=\underline{h}^{1,n-2}$ implies $H^2( \Omega^{n-1}_{\Yhat}(\log E)(-E)) \to H^2( \Omega^{n-1}_{\Yhat}(\log E))$ is an isomorphism.

To prove the claim, let us consider the long exact sequence $$ H^1( \Omega^{n-1}_{\Yhat}(\log E)|_E)\to H^2( \Omega^{n-1}_{\Yhat}(\log E)(-E)) \to H^2( \Omega^{n-1}_{\Yhat}(\log E))\to H^2( \Omega^{n-1}_{\Yhat}(\log E)|_E)$$ and note that the last term $$H^2( \Omega^{n-1}_{\Yhat}(\log E)|_E)\cong \bigoplus\limits_{x \in \mathrm{Sing}(Y)} H^2( \Omega^{n-1}_{\Xhat_x}(\log E_x)|_{E_x})= \bigoplus\limits_{x \in \mathrm{Sing}(Y)} \Gr_{F}^{n-1} H^{n+1}(L_x).$$ But since the singularities are lci,  $H^{n+1}(L_x)=0$ and so $H^2( \Omega^{n-1}_{\Yhat}(\log E)|_E)=0.$ Therefore, we have a surjection $H^2( \Omega^{n-1}_{\Yhat}(\log E)(-E)) \sto H^2( \Omega^{n-1}_{\Yhat}(\log E)).$ Now notice that $$\dim  H^2( \Omega^{n-1}_{\Yhat}(\log E)(-E))= \dim \HH^2 (\DB^{n-1}_Y)= \underline{h}^{n-1,2}$$ and similarly by Serre duality $\dim H^2( \Omega^{n-1}_{\Yhat}(\log E))= \dim H^{n-2}(\Omega^{1}_{\Yhat}(\log E)(-E))= \underline{h}^{1,n-2}.$ By our assumption on the equality of the Hodge-Du Bois numbers, $$H^2( \Omega^{n-1}_{\Yhat}(\log E)(-E)) \sto H^2( \Omega^{n-1}_{\Yhat}(\log E))$$ must then be an isomorphism.

Using the claim, we obtain that the morphism $$\bigoplus\limits_{x \in \mathrm{Sing}(Y)} H^2_{E_x}(\Omega^{n-1}_{\Xhat_x}(\log E_x)(-E_x)) \xto{h'} H^2( \Omega^{n-1}_{\Yhat}(\log E)(-E))$$ is zero.

Consider the composition $$T=\TT^1_Y= H^1(\Omega_{V}^{n-1}) \to T^1_{Y,x}\sto H_{E_x}^2(\Omega_{\Xhat_x}^{n-1}(\log E_x)(-E_x)).$$

For each non-1-rational singular point $x$ consider $\beta'_x\in H^2_{E_x} (\Omega_{\Xhat_x}^{n-1}(\log E_x)(-E_x))$ but such that $\beta'_x\notin m_xH^2_{E_x} (\Omega_{\Xhat_x}^{n-1}(\log E_x)(-E_x)).$ 
Now, $h'$ is zero. Hence, the element $\beta'=\bigoplus\limits_{x \in \mathrm{Sing}(Y)} \beta'_x$ comes from a global first-order deformation and, therefore, 
$$\text{im}(H^1(\Omega_{V}^{n-1}) \to H_{E_x}^2(\Omega_{\Xhat_x}^{n-1}(\log E_x)(-E_x))) \not\subset m_x H_{E_x}^2(\Omega_{\Xhat_x}^{n-1}(\log E_x)(-E_x))
.$$

But then looking at the composition $$\TT^1_Y= H^1(\Omega_{V}^{n-1}) \to T^1_{Y,x}\sto H_{E_x}^2(\Omega_{\Xhat_x}^{n-1}(\log E_x)(-E_x))$$ we deduce the image of $\TT^1_Y$ in $T^1_{Y,x}$ is not contained in $m_xT^1_{Y,x}$ when $(Y,x)$ is not 1-rational.

We begin by showing (2). Under our assumption for any singularity $(Y,x)$ one has $\TT^1_Y \twoheadrightarrow T^1_{Y,x}/m_xT^1_{Y,x}$ as in the proof of Theorem \ref{thm: fanothm}. By Friedman's unobstructedness Theorem \ref{thm: unobstructed}, $\text{Def}(Y)$ is smooth. Hence we may apply Lemma \ref{lemma: smoothinglci} with $S'=\text{Def}(Y)$ and $T=\TT^1_Y$, and from the fact that $\TT^1_Y \twoheadrightarrow T^1_{Y,x}/m_xT^1_{Y,x}$ we conclude that $Y$ is smoothable.

We now show (1) i.e. that a projective $Y$ can always be deformed to a Calabi-Yau variety whose singularities are 1-rational. The proof proceeds by iteration exactly as in the proof of Theorem \ref{thm: fanothm} and of Theorem \ref{thm: lcinot1DB}. The only thing left to check is that after deforming one can keep applying the above procedure to be able to deform again. 

To continue the iteration, it suffices to show that the equality of Hodge-Du Bois numbers persists on sufficiently nearby fibers of each chosen one-parameter deformation. Let $f:\mathcal{Y}\to \Delta$ be a one parameter deformation of $Y_0=Y$ and note that  since $H^2(\mathcal{O}_Y)=0$ an ample $L$ on $Y$ extends. So (up to shrinking the disk) this is a projective morphism \cite[Proposition 1.4]{MR946250}.

Now consider the nearby/vanishing cycles long exact sequence $$\dots\to H^k(Y)\to H^k_{\text{lim}}\to H^k_{\text{van}} \to H^{k+1} (Y) \to \dots .$$ Since $Y$ has isolated singularities $H^k_{\text{van}}=\oplus_i \tilde{H}^k(F_i), $ where $F_i$ is a local nearby fiber $F_i=B_{\varepsilon}(x_i)\cap Y_t$ for each singular point $x_i\in Y.$ Note moreover since $\mathcal{Y}$ is lci that
$\tilde{H}^k(F_i)=0$ for $k\neq n.$ This follows from the fact that $(R\phi_f\mathbb Q_{\mathcal Y})[n]$ is perverse and, since the central fiber has isolated singularities, is supported on finitely many points so it is a sheaf (see for example \cite[Proposition 5.5(ii) and Remark 5.6]{MR4292232}).

Therefore, $H^{n-1}(Y)\cong H^{n-1}_{\lim}$ and $H^{n+1}(Y)\twoheadrightarrow H^{n+1}_{\lim}.$ By strictness and using the fact that $\dim \Gr_F^pH^k_{\lim}=\dim \Gr_F^p H^k(Y_t)$ (see for example \cite[Proof of Lemma 1.8]{MR782235}) we conclude that $\underline{h}^{1,n-2}(Y_t)=\underline{h}^{1,n-2}(Y)$
 and $\underline{h}^{n-1,2}(Y_t)\leq \underline{h}^{n-1,2}(Y).$

We now claim $\underline{h}^{1,n-2}(Y_t)\leq \underline{h}^{n-1,2}(Y_t).$ This follows by the same argument as in \cite[Proof of Theorem E]{park2025q} since $Y_t$ has isolated lci singularities,
so the comparison with intersection cohomology gives $H^{n-1}(Y_t,\mathbb{Q})\hookrightarrow IH^{n-1}(Y_t,\mathbb{Q})$ and 
$H^{n+1}(Y_t,\mathbb{Q})\simeq
IH^{n+1}(Y_t,\mathbb{Q})$ 
as mixed Hodge structures and one concludes by intersection cohomology duality and again by strictness.

Putting all the inequalities together we have $$\underline{h}^{1,n-2}(Y)=\underline{h}^{1,n-2}(Y_t)\leq \underline{h}^{n-1,2}(Y_t)\leq \underline{h}^{n-1,2}(Y)$$ and since by assumption the equality of Hodge-Du Bois numbers $\underline{h}^{n-1,2}(Y)=\underline{h}^{1,n-2}(Y)$ holds for $Y$ it must also hold for $Y_t.$

\end{proof}

\begin{remark}
In the previous theorem, if $n\geq 4$ one can check that the condition $\underline{h}^{n-1,2}=\underline{h}^{1,n-2}$ is in fact equivalent to the morphism $H^2_E(\Omega^{n-1}_{\Yhat}(\log E)(-E))\to H^2( \Omega^{n-1}_{\Yhat}(\log E)(-E))$ being zero. Indeed, assuming the morphism is zero we have that dually the morphism 
$$
H^{n-2}(\Omega_{\Yhat}^1(\log E)) \rightarrow 
\bigoplus\limits_{x \in \mathrm{Sing}(Y)} H^{n-2}(\Omega^{1}_{\Xhat_x}(\log E_x ))
$$
is zero. 

Consider the following hypercohomology spectral sequence 

$$E_2^{p,q}=H^p(R^q \pi_{*} \Omega_{\Yhat}^1 (\log E)) 
\Rightarrow H^{p+q}
(\Omega_{\Yhat}^1 (\log E)).$$

Therefore $E_2^{p,q}=0$ whenever both $p,q>0.$

We know
\begin{align*}
H^{n-2}(\Omega_{\Yhat}^1(\log E))
\rightarrow E_2^{0,n-2}
&= H^0(R^{n-2}\pi_*\Omega^1_{\Yhat}(\log E)) \\
&= \bigoplus\limits_{x \in \mathrm{Sing}(Y)} H^{n-2}(\Omega^{1}_{\Xhat_x}(\log E_x))
\end{align*}
is zero. In other words, we have $E_\infty^{0,n-2}=0.$ 

Note $E_2^{n-2,0}=E_\infty^{n-2,0}$ which follows from
$$
E_2^{0,n-3}= \bigoplus\limits_{x \in \mathrm{Sing}(Y)} H^{n-3}(\Omega^{1}_{\Xhat_x}(\log E_x )) =0
$$ by Remark \ref{rmk: nrvanish}, as the singularities are isolated lci. 

Hence, we obtain an isomorphism $E_\infty^{n-2,0}=H^{n-2}(\pi_*\Omega^1_{\Yhat}(\log E))\cong H^{n-2}
(\Omega_{\Yhat}^1 (\log E)).$ But since $n\geq 4$ and the singularities are isolated $$H^{n-2}(\pi_*\Omega^1_{\Yhat}(\log E)(-E))\cong H^{n-2}
(\pi_*\Omega_{\Yhat}^1 (\log E))$$ which combined with the previous result gives the equality $\underline{h}^{n-1,2}=\underline{h}^{1,n-2}.$

\end{remark}

\begin{remark}
In higher dimensions, the condition $\underline{h}^{n-1,2}=\underline{h}^{1,n-2}$ would be immediate if all the singularities were 1-rational (\cite[Corollary 3.22, Theorem 3.24]{friedman2024higher} and \cite[Theorem A]{park2025hodge}), in which case 
$$
\underline{h}^{p,q}(Y)=\underline{h}^{q,p}(Y)=\underline{h}^{\,n-p,\,n-q}(Y)=\underline{h}^{\,n-q,\,n-p}(Y),\qquad
\text{for all } 0\le p\le 1 \text{ and } 0\le q\le n.$$

However, this is too strong an assumption for smoothability, as our arguments rely on the fact that the singularities are not 1-rational i.e., that $H^2_{E_x} (\Omega_{\Xhat_x}^{n-1}(\log E_x)(-E_x))\neq 0.$ 
\end{remark}

\begin{remark}\label{rem: PD}
We conclude with a discussion of the equality $\underline{h}^{n-1,2}=\underline{h}^{1,n-2}$ and its relation to the Poincar\'{e} duality of the variety.

For threefolds, the assumption $\underline{h}^{2,2}=\underline{h}^{1,1}$ is also implied by the Poincar\'{e} duality of the threefold. In fact, if the singularities are rational it is equivalent to the vertical symmetry of the Hodge-Du Bois diamond.

To see this, since the singularities are rational, the border of the diamond is symmetric (\cite[Corollary 3.22, Theorem 3.24]{friedman2024higher} and \cite[Theorem A]{park2025hodge}). More precisely, we have $\underline{h}^{3,3}=\underline{h}^{0,0}$ and $$\underline{h}^{3,1}=\underline{h}^{1,3}=\underline{h}^{2,0}=\underline{h}^{0,2}$$ 
$$\underline{h}^{3,2}=\underline{h}^{2,3}=\underline{h}^{1,0}=\underline{h}^{0,1}.$$ 

Hence, one has full vertical symmetry of the Hodge-Du Bois diamond of $Y$ if and only if $\underline{h}^{2,2}=\underline{h}^{1,1}$ if and only if $Y$ is $\QQ$-factorial.

For fourfolds with rational singularities, the condition $\underline{h}^{3,2}=\underline{h}^{1,2}$ is still equivalent to the full vertical symmetry of the Hodge-Du Bois diamond. 
As before, rational singularities imply that the border of the Hodge-Du Bois diamond is symmetric.  Moreover, the top part of the Hodge-Du Bois diamond ($ p + q > 4$), coincides with intersection
cohomology since the singularities are isolated. Since the singularities are lci, we have $\text{lcdef}(Y)=0,$ and so, by \cite[Corollary 8.9]{park2025hodge} the Hodge structure on $H^{3}(Y, \mathbb{Q})$ is pure and we have
\[
H^{i}(Y, \mathbb{Q}) \simeq IH^{i}(Y, \mathbb{Q})
\quad \text{for } i \in \{1,2\}.
\]
 Therefore, except for the middle row, one has horizontal symmetry. Moreover, vertical symmetry is equivalent to $b_3(Y)= b_5(Y)$, which, by the horizontal symmetry, is equivalent to our condition  $\underline{h}^{3,2} =\underline{h}^{1,2}.$

Of course, in higher dimension the condition $\underline{h}^{n-1,2}=\underline{h}^{1,n-2}$  is still implied by (but weaker than) the vertical symmetry of the Hodge-Du Bois diamond (i.e. by $b_{n-1}(Y)=b_{n+1}(Y)$).

\end{remark}

\vspace{1em}

Let us now explain how the methods above can be used to improve Theorem \ref{thm: FL}. First, we set some notation. Let $L_x$ denote the link of a singular point $x\in Y$ and consider $\pi:\Yhat\to Y$ a strong log resolution of $Y.$ There is a natural map
$$\varphi:\bigoplus\limits_{x \text{ is }1-\text{liminal}} \Gr^{n-1}_FH^n(L_x)  \to  H^{2}(\Omega_{\Yhat}^{n-1})$$ which we describe in the course of the proof. 

\begin{theorem}\label{thm: 1limlci}
Suppose $Y$ is a singular Calabi-Yau variety of dimension $n$ with isolated Du Bois lci singularities. Assume that every non-1-Du Bois singularity satisfies $(\star_1)$ (which is vacuous in the hypersurface case) and that every 1-liminal singularity is a hypersurface singularity.

Suppose there exists an element
$$
\alpha=(\alpha_x)\in
\bigoplus\limits_{x \text{ is }1-\text{liminal}}
\Gr^{n-1}_FH^n(L_x)
$$
which is nonzero in each factor and such that
$\varphi(\alpha)=0$ in $H^2(\Omega_{\Yhat}^{n-1})$.
Then the following hold:
\begin{enumerate}
    \item $Y$ can be deformed to a Calabi-Yau variety whose singularities are 1-rational.
    \item If none of the singularities of $Y$ are 1-rational, then $Y$ is smoothable, and every small smoothing is Calabi-Yau.
\end{enumerate}
\end{theorem}

\begin{proof}[Proof of Theorem \ref{thm: 1limlci}]
First, for each $x$ which is not 1-Du Bois, as in the proof of Theorem \ref{thm: lcinot1DB} we have a surjection
$$\TT^1_Y\to H^2_{E_x}(\Omega_{\Xhat_x}^{n-1}(\log E_x)).$$
By our assumption $(\star_1)$ we therefore have
$$\TT^1_Y\twoheadrightarrow T^1_{Y,x}/m_xT^1_{Y,x}.$$

Then for smoothing of 1-liminal points we make use of the assumption. We note that $\Gr^{n-1}_FH^n(L_x)= H^1(\Omega_{\Xhat_x}^{n-1}(\log E_x)|_{E_x})$ so $\varphi$ is given as the composition of coboundary map $$\bigoplus\limits_{x \text{ is }1-\text{liminal}}H^1(\Omega_{\Xhat_x}^{n-1}(\log E_x)|_{E_x})\to H^2( \Omega^{n-1}_{\Yhat}(\log E)(-E))$$ with the map $H^2( \Omega^{n-1}_{\Yhat}(\log E)(-E))\to H^2( \Omega^{n-1}_{\Yhat}).$ The latter is an injection \cite[Lemma 1.3]{FLklim} so our assumption says
there exists an element $\alpha$ in  $\bigoplus\limits_{x \text{ is }1-\text{liminal}} \Gr^{n-1}_FH^n(L_x)$ which is nonzero in each factor and which maps to zero in $H^2( \Omega^{n-1}_{\Yhat}(\log E)(-E)).$ Note for 1-liminal singularities $H^2_{E_x}(\Omega^{n-1}_{\Xhat_x}(\log E_x))=0$ as they are 1-Du Bois so the short exact sequences:
$$0\to \text{Gr}_F^{n-1}H^n(L_x) \to H^2_{E_x}(\Omega^{n-1}_{\Xhat_x}(\log E_x)(-E_x))\to H^2_{E_x}(\Omega^{n-1}_{\Xhat_x}(\log E_x))\to 0,$$ implies $\text{Gr}_F^{n-1}H^n(L_x) \cong H^2_{E_x}(\Omega^{n-1}_{\Xhat_x}(\log E_x)(-E_x)).$ So using the notation as in the Proof of Theorem \ref{thm: lcigeneral} our assumption says there exists an element $\alpha\in H^2_E( \Omega^{n-1}_{\Yhat}(\log E)(-E))$ which is nontrivial in each 1-liminal component such that $h'(\alpha)=0$ in $H^2( \Omega^{n-1}_{\Yhat}(\log E)(-E)).$ 
By the short exact sequence at the end of Section \ref{sec: 2.3} $m_xH^2_{E_x}(\Omega^{n-1}_{\Xhat_x}(\log E_x)(-E_x))=0$ for 1-liminal singularities so we have $\alpha\in H^2_E( \Omega^{n-1}_{\Yhat}(\log E)(-E))$ whose component $\alpha_x$ does not belong to $m_xH^2_{E_x}(\Omega_{\Xhat_x}^{n-1}(\log E_x)(-E_x))$ for each 1-liminal component and also $h'(\alpha)=0.$ So, as in Proof of Theorem \ref{thm: lcigeneral}
we can find $\theta_2 \in \mathbb{T}^1_Y$ which maps to an element outside of $m_xT^1_{Y,x}$ for each $x$ which is 1-liminal. 

Since every 1-liminal point is a hypersurface singularity by assumption, $T^1_{Y,x}/m_xT^1_{Y,x}$ is one-dimensional at such a point. Hence, from the fact that $\theta_2$ maps outside of $m_xT^1_{Y,x}$ we conclude that
$$\TT^1_Y\twoheadrightarrow T^1_{Y,x}/m_xT^1_{Y,x}.$$
Thus this map is surjective for every $x$ which is not 1-rational.

By Friedman's unobstructedness Theorem \ref{thm: unobstructed}, $\text{Def}(Y)$ is smooth. Hence we may apply Lemma \ref{lemma: smoothinglci} with $S'=\text{Def}(Y)$ and $T=\TT^1_Y$ to conclude that a sufficiently small general deformation smooths all the non-1-rational points. Since the 1-rational singularities remain 1-rational after a small deformation, this proves (1).

If none of the singularities of $Y$ are 1-rational, the resulting deformation is a smoothing. Moreover, every small smoothing is Calabi-Yau as in the proof of Theorem \ref{thm: lcinot1DB}, which proves (2).
\end{proof}

\printbibliography
\Addresses

\end{document}